\documentclass{amsart}

\usepackage[enumitem,theorem,Rn,hyperref]{paper_diening}

\usepackage{a4wide}
\usepackage{microtype}

\providecommand{\osc}{\operatorname{osc}}

\providecommand{\epsilonDG}{\ensuremath{\epsilon_{\mathrm{DG}}}}

\newcommand*{\sep}{\,|\,}

\renewcommand{\rho}{\varrho}

\title{Higher Order Calder{\'o}n-Zygmund Estimates for the $p$-Laplace Equation}

\author{Anna Kh.~Balci}

\author{Lars Diening}

\author{Markus Weimar}

\begin{document}

\begin{abstract}
  The paper is concerned with higher order \Calderon{}-Zygmund
  estimates for the $p$-Laplace equation
  \begin{align*}
    -\divergence(A(\nabla u))
    &:= -\divergence{(\abs{\nabla
      u}^{p-2}\nabla u)}=-\divergence F, \qquad 1<p<\infty.
  \end{align*}
  We are able to transfer local interior Besov and Triebel-Lizorkin
  regularity up to first order derivatives from the force term~$F$ to
  the flux~$A(\nabla u)$. For $p\geq 2$ we show that
  $F \in \bfB^s_{\rho,q}$ implies $A(\nabla u) \in \bfB^s_{\rho,q}$
  for any $s \in (0,1)$ and all reasonable $\rho,q \in (0,\infty]$ in
  the planar case. The result fails for~$p<2$.  In case of higher
  dimensions and systems we have a smallness restriction on~$s$. The
  quasi-Banach case $0<\min\{\rho,q\} < 1$ is included, since it has
  important applications in the adaptive finite element analysis.  As an
  intermediate step we prove new linear decay estimates for
  $p$-harmonic functions in the plane for the full range $1<p<\infty$.
\end{abstract}

\thanks{The research of Anna Kh. Balci was partly supported by DAAD, RFBR and  the Ministry of Education and Science of the Russian Federation (projects 19-01-00184 and 1.3270.2017/4.6).}

\keywords{p-Laplacian; nonlinear elliptic equations; regularity of solutions}

\subjclass[2010]{%
35J92, 
46E35, 
65M99, 
35J60, 
35B65 
}

\maketitle

\numberwithin{theorem}{section}
\numberwithin{equation}{section}

\section{Introduction}
\label{sec:introduction}
In this paper we study \Calderon{}-Zygmund type estimates for the weak
solution of the $p$-Poisson equation
\begin{align}
  \label{eq:plap}
  -\divergence(A(\nabla u))
  &:= -\divergence{(\abs{\nabla
    u}^{p-2}\nabla u)}=-\divergence F \qquad
    \text{in $\Omega$}
\end{align}
where $d,n\in\setN$, $\Omega$ is an open set in $\setR^d$,
$1<p<\infty$, and~$u \colon \Omega\to \setR^n$ is the unknown. All our
results are of local nature so no boundary conditions are
required. Most of them are restricted to $p\geq 2$ and scalar
solutions ($n=1$) for $d=2$.

The main objective in non-linear \Calderon{}-Zygmund theory is to
transfer the regularity of the right hand side~$F$ to the
flux~$A(\nabla u)$ (or to~$\nabla u$ itself)
in the norm of an
appropriate function space~$\bfX$. The corresponding
estimate can be written as
\begin{align}
  \label{eq:Calderon-Zygmund-global}
  \norm{A(\nabla u)}_{\bfX} \le C\,\norm{F}_{\bfX},
\end{align}
or, in its local version,
\begin{align}
  \label{eq:Calderon-Zygmund}
  \norm{A(\nabla u)}_{\bfX(B)} \le C\,\norm{F}_{\bfX(2B)} +
  \texttt{lower order terms of~$A(\nabla u)$},
\end{align}
where $B$ denotes an arbitrary ball such that $2B\subset \Omega$.

The choice~$\bfX=L^{p'}$ (with $\frac{1}{p} + \frac{1}{p'}=1$)
corresponds to the standard estimates of weak solutions. The first
breakthrough was the result of~\cite{IwaMan89}, who showed the
estimate~\eqref{eq:Calderon-Zygmund-global} for~$\bfX = L^r$ and
all~$r \in [p',\infty)$. Later on this result was extended
to~$\bfX=\setBMO$ for~$p\ge 2$ in \cite{DiBMan93} and for an arbitrary
exponent~$p>1$ in \cite{DieKapSch11}. It became clear from the
calculations in~\cite{DieKapSch11} that it is better to look at the
mapping~$F \mapsto A(\nabla u)$ rather than $F \mapsto \nabla u$. This
is also supported by~\cite{KuuMin13}, where potential estimates for
the mapping~$f := \divergence F \mapsto A(\nabla u)$ have been
studied.  Moreover, it has been shown in~\cite{DieKapSch11} that it is
possible to take $\bfX= \setBMO_\omega$, or $\bfX = C^{0,\alpha}$,
resp., as long as the modulus of continuity~$\omega$,
resp.~$\alpha>0$, satisfies some smallness condition which depends on
the best known regularity of $p$-harmonic functions. In particular,
for~$d \geq 3$ or vectorial solutions the exponent~$\alpha>0$ is just
an unknown small quantity.

In this paper we extend the \Calderon-Zygmund estimates for~$p\geq 2$
and~$d=2$ to spaces of differentiability up to order one. In
particular, we show that the estimate~\eqref{eq:Calderon-Zygmund}
holds true also for Besov spaces $\bfX=\bfB^s_{\rho,q}$ for all
exponents of smoothness~$s \in (0,1)$, every integrability
parameter~$\rho \in (0,\infty]$, and all fine
indices~$q \in (0,\infty]$ such
that~$\bfB^s_{\rho,q} \embedding\embedding L^{p'}$. Moreover, if
additionally $\rho<\infty$, then a similar assertion remains valid in
the scale of Triebel-Lizorkin spaces $\bfX=\bfF^s_{\rho,q}$.  We refer
to Theorem~\ref{thm:reg-transfer} for the precise statements.  Let us
stress the fact that these scales include a lot of classical function
spaces such as, e.g., H\"older-Zygmund, Bessel-potential, or
Sobolev-Slobodeckij spaces, as special cases \cite{RunSic96}.  The
restriction~$p\geq 2$ in our result is natural in this context, since
the assertion fails for~$1<p<2$ even for~$F\equiv 0$ and~$d=2$, see
Subsection~\ref{ssec:duality}. The assumed compact embedding in
$L^{p'}$ ensures that we are in the context of weak solutions, i.e.,
that $u \in W^{1,p}_{\loc}(\Omega)$.  In the case~$d \geq 3$ we obtain
similar results, but then there are restrictions on~$s$ due to some
open problems (see Subsection~\ref{ssec:open-problem}) on the
regularity of $p$-harmonic functions in higher dimensions.

Our work is motivated by the numerical analysis of the $p$-Poisson
equation using wavelets or the adaptive finite element method. Note
that the approximability of the solutions by discrete ones is
determined solely by the differentiability~$s$
from~$\bfB^s_{\rho,q}$. We refer to~\cite{CohDahDeV01,DeV98} for a
detailed study of numerical approximability. However, in many cases it
is possible to increase~$s$ by decreasing the integrability~$\rho$,
where the strongest results are obtained if we take $\rho<1$.  Then we
are in the regime of quasi-Banach spaces, but nevertheless also in
this case the smoothness~$s$ still determines the rates of
convergence of best $N$-term approximations.
For this reason it is important that our
estimates cover the full range of parameters $\rho,q\in(0,\infty]$.

Other authors also investigated estimates for~$A(\nabla u)$ in
terms of Sobolev or Besov spaces. For example, Cianchi and Maz'ya have
shown in~\cite{CiaMaz18} that $F \in W^{1,2}$, so $f :=
\divergence F \in L^2$, implies that~$A(\nabla u) \in W^{1,2}$
for any~$d \geq 2$ and any~$1<p<\infty$. They also obtain global
results under minimal conditions on~$\partial \Omega$. Moreover, it
has been shown by Avelin, Kuusi, and Mingione \cite{AveKuuMin18}
that~$f \in L^1$ implies that locally $A(\nabla u) \in W^{s,1}$
for any $s \in (0,1)$, $d\geq 2$ and $p > 2 - \frac
1d$. For~$p\leq d$ this requires the concept of so-called \emph{solutions obtained as limits of approximations} (SOLA). 
Both results support the fact that the mappings
$F \mapsto A(\nabla u)$ and $f \mapsto A(\nabla u)$ are the natural
ones. Our regularity results differ from~\cite{CiaMaz18}
and~\cite{AveKuuMin18} in the sense that we provide estimates for all
integrability exponents~$\rho$ (from $\bfB^s_{\rho,q}$ or
$\bfF^s_{\rho,q}$), while~\cite{CiaMaz18} is restricted to~$\rho=2$
and~\cite{AveKuuMin18} is restricted to~$\rho=1$. Let us mention again
that estimates for arbitrary exponents~$\rho$ are only possible
for~$p \geq 2$, see Subsection~\ref{ssec:duality}.

In Subsection~\ref{ssec:transfer-nabla-u} we discuss how our results
on the regularity of~$A(\nabla u)$ translate into regularity
assertions for~$\nabla u$ and $V(\nabla u)=\abs{\nabla u}^{\frac{p-2}{2}} \nabla u$. This allows us also to
compare our results with the work of other authors on the higher
differentiability of these quantities.  For example, it has been shown
in~\cite{DahDieHarSchWei16} that for $d=2$ and~$f \in L^\infty$ there
holds $u \in \bfB^{s}_{\rho,\rho}$ for all $s \in (0,2)$ if $p \leq 2$
and all $s \in (0,p')$ if $p > 2$, where in both cases
$\rho=\rho(d,s,p) > 0$ refers to the adaptivity scale of $L^p$, i.e.,
$s-\frac{d}{\rho}=-\frac{d}{p}$.  These results also hold globally on
Lipschitz domains with zero boundary data. Corner regularity results
with strong conditions on the right-hand side for~$d=2$ have been
studied in \cite{HarWei18}.  The $C^{0,\alpha}$-regularity
of~$A(\nabla u)$ up to the boundary for smooth domains has been
studied in~\cite{BreCiaDieKuuSch17b}, however it rules out the case of
polygonal domains that appear in the context of finite elements.
Moreover, is has been shown in~\cite{CloGioPas17} that for $p\geq 2$,
$d \geq 2$ and $s \in (0,1)$ a forcing term~$F \in \bfB^s_{2,q}$
implies that
$V(\nabla u) \in
\bfB^{\frac{p' s}{2}}_{2,\frac{2q}{p'}}$ for
$1 \leq q \leq \frac{2 d}{d-2s}$. For a more detailed comparison of
our results to that from~\cite{DahDieHarSchWei16} and
\cite{CloGioPas17} we refer to Subsection~\ref{ssec:transfer-nabla-u}.

The main idea of our proof is to employ a well-known characterization
of Besov and Triebel-Lizorkin spaces in terms of oscillations, see
Lemma~\ref{lem:characterization}. This allows to reduce the proof
of~\eqref{eq:Calderon-Zygmund} to an oscillation decay estimate for
$A(\nabla u)$. This fundamental decay estimate is formulated in
Theorem~\ref{thm:osc-estimates}. It is of independent interest since
it allows to significantly improve the decay estimate from
\cite{BreCiaDieKuuSch17} at least in the case of the plane.

Certainly, the oscillation estimates for~$A(\nabla u)$ can never be
better than the ones for $p$-harmonic functions, i.e., for~$h$ with
$\divergence(A(\nabla h))=0$, which corresponds to the case~$F\equiv 0$.  In
two dimensions we are able to prove a new (almost) linear decay estimate for the 
oscillations of~$A(\nabla h)$ for~$p \geq 2$. Indeed, in
Theorem~\ref{thm:h-decay} we show that
\begin{align*}
  \dashint_{\theta B} \abs{A(\nabla h) - \mean{A(\nabla h)}_{\theta
  B}} \,dx
  &\leq c_\beta\, \theta^\beta
    \dashint_{B} \abs{A(\nabla h) - \mean{A(\nabla h)}_{
    B}} \,dx
\end{align*}
for any~$\theta,\beta \in (0,1)$. From this we deduce by duality decay
estimates for $\nabla h$ in the case~$1 < p \leq 2$, see
Theorem~\ref{thm:h-decay-grad}.  It has been shown by Iwaniec and Manfredi in  \cite{IwaMan89}
that~$A(\nabla h) \in C^1$ for~$p\geq 2$ while $\nabla h \in C^1$ if
$p \leq 2$. However, the techniques therein do not provide qualitative decay estimates. Instead, we use and improve the approach of~\cite{AraTeiUrb17} and~\cite{BaeKov05}, which allows us to
obtain new decay estimates. We also implement several ideas
from~\cite{BreCiaDieKuuSch17}.
    
The paper is organized as follows: In Section~\ref{sec:regul-p-harm}
we study the regularity of~$p$-harmonic functions in the plane. Here we
deduce the important decay estimates for $A(\nabla h)$ that we shall need
later. Starting from Section~\ref{sec:oscill-estim} we study the
$p$-Poisson equation with a force term~$\divergence F$. We derive in
this section the crucial oscillation estimates of~$A(\nabla u)$ in
terms of the oscillations of~$F$. In
Section~\ref{sec:regularity-transfer} we prove the nonlinear
\Calderon-Zygmund{} estimates that allow to transfer
$\bfB^{s}_{\rho,q}$, resp. $\bfF^s_{\rho,q}$ regularity from~$F$ to
$A(\nabla u)$. Here we also explain how the regularity of~$A(\nabla u)$ implies regularity of~$\nabla u$ and $V(\nabla u)$.  Throughout the paper we
assume~$p\geq 2$. Only in Subsection~\ref{ssec:duality} we deal with
the case~$1< p<2$ and present a new decay estimate.

\section{\texorpdfstring{Regularity of $p$-harmonic Functions}{Regularity of p-harmonic functions}}
\label{sec:regul-p-harm}
Regularity studies of solution to the  problem~\eqref{eq:pharmonic} are about 50 years
old. They go back to Ural'tseva \cite{Ura68}, where it was shown
that~$p$-harmonic functions belong to the local H\"older
class~$C_{\loc}^{1,\alpha}(\Omega)$ for some
exponent~$\alpha=\alpha(d,p)<1$. For the case~$d=2$ the sharp value of
the H\"older exponent~$\alpha$ is known, see \cite{IwaMan89}, while for $d\geq 3$ this problem is still open.

Before we proceed let us first introduce some notation.  For vectors $Q$ we define~$A$
and $V$ in the following way:
\begin{align*}
  A(Q) &=|Q|^{p-2}Q, 
  \\
  V(Q) &=|Q|^{\frac{p-2}{2}}Q,
\end{align*}
where $\abs{\cdot}$ denotes the Euclidean norm.
Note that~$A$ and $V$ are isomorphisms. Moreover, by~$B$, $B_r$, and
$B_r(x)$ we usually denote open Euclidean balls with radius~$r>0$ and center~$x\in \Omega\subset\setR^d$. We write $\lambda B$ for the ball with same center as~$B$ but scaled in
size by~$\lambda$.
Further, for~$f\in L^1_{\loc} (\setR^d)$ we define the mean value over the ball $B$ as
\begin{align*}
	\mean{f}_B:= \dashint_B f\, dx,
\end{align*}
where $\dashint_B \dotsm \,dx := \abs{B}^{-1} \int_B \dotsm \,dx$ denotes the average integral with $\abs{B}$ being the volume of $B$. The same notation is employed also in the vector-valued case.
Moreover, we shall use $c$ as a generic positive constant which may
change from line to line, but does not depend on the crucial
quantities. We will use the notation $f \lesssim g$ if there exist a
constant such that $f \leq c\, g$. Finally, we write $f \eqsim g$ if
$f \lesssim g$ and $g \lesssim f$.

\begin{definition}
  A function $h \colon \Omega \to \setR$ is called $p$-harmonic
  in~$\Omega\subset \setR^d$ if it is a weak solution of the
  $p$-harmonic equation, i.e., $h \in W^{1,p}_{\loc}(\Omega)$ and
  \begin{align}
    \label{eq:pharmonic}
    -\divergence(A(\nabla h)) &= 0
  \end{align}
  in the distributional sense.
\end{definition}
Throughout the paper we will use the letters~$h$ for $p$-harmonic functions
and~$u$ for solutions to the $p$-Poisson equation \eqref{eq:plap}.


The main result of this  section is the following decay estimate for~$A(\nabla h)$.
\begin{theorem}
  \label{thm:h-decay}
  Let $h\colon \Omega \to \setR$ be $p$-harmonic with $p \geq 2$
  on~$\Omega \subset \setR^2$. Then for all $\beta \in (0,1)$, there
  exists~$c_\beta>0$ such that for all balls~$B \subset \Omega$ and
  all $\theta \in (0,1)$ there holds
  \begin{align*}
    \dashint_{\theta B} \abs{A(\nabla h) - \mean{A(\nabla h)}_{\theta
    B}} \,dx
    &\leq c_\beta\, \theta^\beta
      \dashint_{B} \abs{A(\nabla h) - \mean{A(\nabla h)}_{
      B}} \,dx.
  \end{align*}
\end{theorem}
In Proposition~\ref{pro:h-decay-pprime} below we present a
corresponding estimate with power~$p'$ on the left-hand side.
 
The proof of Theorem~\ref{thm:h-decay} requires a few preliminary steps. The basic
idea is to distinguish between the \emph{non-degenerate} and the
\emph{degenerate} case. The non-degenerate case is the one where
\begin{align*}
  \dashint_B \abs{V(\nabla h) - \mean{V(\nabla h)}_B}^2\,dx &\leq \epsilonDG
  \dashint_B \abs{V(\nabla h)}^2\,dx
\end{align*}
for a suitable small~$\epsilonDG>0$.  In particular, $V(\nabla h)$ is
(in average) close to the constant~$\mean{V(\nabla h)}_B$,
so~$\nabla h$ is also close to a constant. In this case
$A(\nabla h) \approx \mean{\nabla h}_B^{p-2} \nabla h$, so the
equation behaves locally like a linear equation with constant
coefficients and we get our decay estimates from this. See Subsection~\ref{ssec:non-degenerate-case-h} for details. In contrast, for the degenerate case we have to argue differently. In this situation
we will use certain decay estimates of quasi-conformal gradient maps which also explain the restriction to $d=2$, see Subsection~\ref{ssec:degenerate-case-h}.

Most of our results are restricted to the case~$p\ge2$. However, 
the following remarkable decay estimate for the case ~$1 < p \leq 2$
is obtained in Subsection~\ref{ssec:duality} by a duality argument.
\begin{theorem}
	\label{thm:h-decay-grad}
	Let $h\colon \Omega \to \setR$ be $p$-harmonic with $1<p\leq 2$
	in~$\Omega \subset \setR^2$. Then for all $\beta \in (0,1)$, there
	exists~$c_\beta>0$ such that for all balls~$B \subset \Omega$ and
	all $\theta \in (0,1)$ there holds
	\begin{align*}
		\dashint_{\theta B} \abs{\nabla h - \mean{\nabla h}_{\theta
				B}} \,dx
		&\leq c_\beta\, \theta^\beta
		\dashint_{B} \abs{\nabla h - \mean{\nabla h}_{
				B}} \,dx.
	\end{align*}
\end{theorem}

\begin{remark}
	\label{rem:h-decay-Rd}
	The Theorems~\ref{thm:h-decay} and \ref{thm:h-decay-grad}
	improve the decay results from \cite[Remark~5.6]{DieKapSch11}
	significantly in the situation of the plane. Indeed, the result
	in~\cite{DieKapSch11} is restricted to~$\beta \in (0,\beta_0)$,
	where~$\beta_0>0$ is some unknown small number.
\end{remark}

\subsection{Shifted Orlicz functions and monotonicity}
\label{ssec:shift-orlicz-funct}

In this subsection we introduce shifted N-functions and present some
monotonicity estimates.

For $1<p<\infty$ 
we define
$\phi\colon [0,\infty) \to [0,\infty)$ by
\begin{align*}
  \phi(t) &:= \tfrac{1}{p} t^p.
\end{align*}
Then $\phi$ is a so-called N-function, i.e. there exists a derivative~$\phi'$
of~$\phi$ which is right continuous, non-decreasing, and satisfies
$\phi'(0)=0$, as well as $\phi'(t)>0$ for $t>0$. In particular,~$\phi$ is
convex.  By~$\phi^*$ we denote the complementary N-function, i.e.,
\begin{align*}
  \phi^*(t) &:=\sup_{s\ge 0}(st-\phi(s)).
\end{align*}
Especially, in our setting, there holds $\phi^*(t) = \frac{1}{p'}t^{p'}$ with
$\frac 1p + \frac 1{p'} = 1$.

We will use a technique based on the properties of \emph{shifted}
N-functions, introduced in~\cite{DieEtt08,DieKre08}:
\begin{definition}
  For~$a,t\ge 0$ we define the shifted~$N$-functions $\phi_a$ as
  \begin{align*}
    \phi_a(t) := \int_0^t \frac{\phi'(\max \set{a,s})}{\max \set{a,s}} s \, ds.
  \end{align*}
\end{definition}
\begin{remark}
  We use here the version of~\cite{DieForWan17} that is equivalent to
  the original one, where $\max\set{a,s}$ is replaced by $a+s$. This
  new version however has a few simple advantages, e.g.,
  $(\phi_a)^* = \phi^*_{\phi'(a)}$ instead of
  $(\phi_a)^* \eqsim \phi^*_{\phi'(a)}$. Note that in our notation
  the~$*$ binds stronger than the shift index. That is, we let
  $\phi^*_a := (\phi^*)_a$.
\end{remark}

Choosing $\phi$ as above, we have the following equivalences for~$a,t \geq 0$
\begin{align*}
  \phi_a(t) &\eqsim (a+t)^{p-2}t^2,
  \\
  \phi_a'(t) &\eqsim (a+t)^{p-2} t,
  \\
  (\phi_a)^*(t) = \phi^*_{\phi'(a)}(t) &\eqsim (a^{p-1}+t)^{p'-2}t^2.
\end{align*}
It is important to observe, that the family $\set{\phi_a}_a$ satisfies
a uniform~$\Delta_2$ and $\nabla_2$-condition, i.e., uniformly
in~$a,t \geq 0$ there holds $\phi_a(2t) \eqsim \phi_a(t)$ and
$\phi^*_a(2t) \eqsim \phi^*_a(t)$. This implies that Young's
inequality holds independently of the shift, i.e. for every~$\delta>0$
there exists~$c_\delta$ such that for all $s,t,a \geq 0$ there holds
\begin{align*}
  \phi_a'(s)\,t \le \delta \phi_a(s)+c_\delta\phi_a(t).
\end{align*}

In the following auxiliary statements we recall some well-known
connections between~$A$, $V$, and shifted~N-functions. For the proofs
we refer to \cite[Lemma~3]{DieEtt08}, \cite[Appendix]{DieKre08}
and~\cite{DieForWan17}. Here and in what follows $\cdot$ denotes the
  Euclidean scalar product.
\begin{lemma}[Monotonicity]
  \label{lem:hammer}
  For $1<p<\infty$ we have
  \begin{alignat*}{3}
    (A(P)-A(Q))\cdot (P-Q) &\eqsim |V(P)-V(Q)|^2 && \eqsim
    (|Q|+|P|)^{p-2}|Q-P|^2
    \\
    &\eqsim \varphi_{|Q|}(|P-Q|) &&\eqsim
    \phi^*_{\abs{A(Q)}}(|A(P)-A(Q)|),
  \end{alignat*}
as well as
  \begin{equation*}
  	\abs{A(P)-A(Q)} \eqsim \varphi_{|Q|}'(|P-Q|)
  	\qquad\text{and}\qquad 
  	A(Q)\cdot Q =|V(Q)|^2 \eqsim \varphi(|Q|).
  \end{equation*}
\end{lemma}
Using the fact that the function $t \mapsto (\abs{Q} +t)^{p-2}$ is increasing for~$p\geq 2$,
we immediately obtain the following corollary:
\begin{corollary}
  \label{cor:hammer-pge2}
  Let $p \geq 2$, then for all vectors~$P$ and $Q$ there holds
  \begin{align*}
    \abs{P-Q}^{p'} \lesssim (A(P)-A(Q))\cdot(P-Q)
    \lesssim \abs{A(P)-A(Q)}^{p'}
  \end{align*}
  and for all~$t,a\geq 0$ we have $\phi^*_a(t) \leq c\, t^{p'}$.
\end{corollary}

The next lemma is taken from \cite[Lemma~2.5]{DieKapSch11}. It is a refined version from the one in~\cite{DieKre08} and shows how to perform a ``shift-change''.
\begin{lemma}
  \label{lem:shift-change}
  Let $1<p< \infty$. Then for all vectors~$P$, $Q$ and
  every~$\lambda \in (0,1]$ it holds
  \begin{align*}
    \phi_{\abs{P}}(t) 
    	&\le c\,\lambda^{1-\max\{p',2\}}\phi_{\abs{Q}}(t)+\lambda\, |V(P)-V(Q)|^2,
    \\
    \phi^*_{\abs{A(P)}}(t) 
    	&\le c\,\lambda^{1-\max\{p,2\}}\phi^*_{\abs{A(Q)}}(t)+\lambda\, |V(P)-V(Q)|^2,
  \end{align*}
  where the constants only depend on~$p$.
\end{lemma}

Moreover, we will make frequent use of the following well-known estimate.
\begin{lemma}
  \label{lem:osc-aux}
  Let $1 \leq q <\infty$. Then for all balls~$B$ and $G \in L^q(B)$, there
  holds
  \begin{align*}
    \inf_{G_0} \bigg( \dashint_B \abs{G-G_0}^q \,dx\bigg)^{\frac 1q} \leq
    \bigg(\dashint_B \abs{G-\mean{G}_B}^q \,dx\bigg)^{\frac 1q}
    \leq 2\,     \inf_{G_0} \bigg( \dashint_B \abs{G-G_0}^q
    \,dx\bigg)^{\frac 1q},
  \end{align*}
  where the infima are taken over all constants~$G_0$.
  If~$q=2$, then we have equality in the first estimate. It is
  possible to replace~$B$ by an arbitrary set of positive measure.
\end{lemma}

Finally, given a gradient $\nabla v$, we define its $A$- and $V$-averages $\mean{\nabla v}_B^A$, resp.\ $\mean{\nabla v}_B^V$, on balls~$B$ by the relations
\begin{align*}
  A(\mean{\nabla v}_B^A) = \mean{A(\nabla v)}_B
  \qquad\text{and}\qquad
  V(\mean{\nabla v}_B^V) = \mean{V(\nabla v)}_B
\end{align*}
and note that both of them are well-defined, since $A(\cdot)$ and $V(\cdot)$ are isomorphisms.
Then the next result, taken from~\cite[Lemma~A.2]{DieKapSch11}, states that the mean oscillations of~$V(\nabla v)$ with respect to the different versions of averages are equivalent.
\begin{lemma}
  \label{lem:DKS6.2}
  Assume that $\nabla v \in L^p(B)$ for a given ball $B$. Then
  \begin{align*}
    \dashint_B \abs{V(\nabla v) - V(\mean{\nabla v}_B^V)}^2
    &\eqsim
      \dashint_B \abs{V(\nabla v) - V(\mean{\nabla v}_B)}^2\,dx
    \eqsim\dashint_B \abs{V(\nabla v) - V(
      \mean{\nabla v}_B^A)}^2\,dx.
  \end{align*}
\end{lemma}

\subsection{Reverse H\"older's estimate}
\label{ssec:reverse-hold-estim}

In this subsection we show that it is possible to measure the
oscillations of~$A(\nabla h)$ with or without the power~$p'$
for~$p\geq 2$.

Let us begin with the following estimate of reverse H\"older
type.
\begin{lemma}[{\cite[Corollary~3.5]{DieKapSch11}}]
  \label{lem:rev-hoelder}
  If $h$ is $p$-harmonic with~$p\geq 2$, then for all~$Q$
  \begin{align*}
    \dashint_{B} \abs{V(\nabla h) - V(Q)}^2\,dx
    &\lesssim
      \phi^*_{\abs{A(Q)}} \bigg( \dashint_{2B} \abs{A(\nabla h) -
      A(Q)}\,dx \bigg).
  \end{align*}
\end{lemma}
By combining this with Lemma~\ref{lem:hammer}, it follows that
\begin{align*}
  \dashint_{B} \phi^*_{\abs{A(Q)}} \big( \abs{A(\nabla h) -
  A(Q)}\big) \,dx
  &\lesssim
    \phi^*_{\abs{A(Q)}} \bigg( \dashint_{2B} \abs{A(\nabla h) -
    A(Q)}\,dx \bigg).
\end{align*}
If we now apply the inverse of~$\phi^*_{\abs{A(Q)}}$ to both sides,
then we obtain
\begin{align}
  \label{eq:3}
  \big(\phi^*_{\abs{A(Q)}}
  \big)^{-1} \bigg(  \dashint_{B} \phi^*_{\abs{A(Q)}} \big( \abs{A(\nabla h) -
  A(Q)}\big) \,dx  \bigg)
  &\lesssim
    \dashint_{2B} \abs{A(\nabla h) -
    A(Q)}\,dx.
\end{align}
In order to proceed further, we shall need the following auxiliary lemma.
\begin{lemma}
  \label{lem:phi-ast-inv-hoelder}
  If $p\geq 2$, then for all~$Q$ we have
  \begin{align*}
    \bigg(\dashint_B \abs{f}^{p'}\,dx\bigg)^{\frac 1{p'}}
    &\lesssim
      \big(\phi^*_{\abs{A(Q)}} \big)^{-1}\bigg( \dashint_B
      \phi^*_{\abs{A(Q)}}(\abs{f})\,dx\bigg).
  \end{align*}
\end{lemma}
\begin{proof}
  We estimate 
  \begin{align*}
    \psi(t) := \phi^*_{\abs{A(Q)}}\big(t^{\frac{1}{p'}}\big)
    &\eqsim \Big(\abs{A(Q)} +
      t^{\frac {1}{p'}}\Big)^{p'-2} t^{\frac
      2{p'}}
    \eqsim
      \begin{cases}
        \abs{A(Q)}^{p'-2}\, t^{\frac 2{p'}} &\quad \text{for $t^{\frac{1}{p'}} \leq
          \abs{A(Q)}$},
        \\
        t &\quad \text{for $t^{\frac{1}{p'}} > \abs{A(Q)}$}.
      \end{cases}
  \end{align*}
  Now $\psi$ is continuous and convex. Thus, by Jensen's inequality,
  \begin{align*}
    \phi^*_{\abs{A(Q)}}\Bigg( \bigg(\dashint_B
    \abs{f}^{p'}\,dx\bigg)^{\frac 1{p'}} \Bigg)
    &\eqsim \psi\bigg(\dashint_B
    \abs{f}^{p'}\,dx\bigg)
    \geq \dashint_B \psi\Big( \abs{f}^{p'} \Big)\,dx
      \eqsim
      \dashint_B
      \phi^*_{\abs{A(Q)}}(\abs{f})\,dx.
  \end{align*}
  This proves the claim.
\end{proof}
\begin{proposition}[Reverse H\"older]
  \label{pro:A-rev-hoelder}
  Let $h$ be $p$-harmonic with~$p\geq 2$. Then
  \begin{align*}
    \bigg(\dashint_{B} \abs{A(\nabla h) - \mean{A(\nabla h)}_B}^{p'}\,dx \bigg)^{\frac 1{p'}}
    &\lesssim
      \dashint_{2B} \abs{A(\nabla h) - \mean{A(\nabla h)}_{2B}}\,dx.
  \end{align*}
\end{proposition}
\begin{proof}
  We define~$Q$ by $A(Q) :=\mean{A(\nabla h)}_{2B}$. Then
  the combination of~\eqref{eq:3} and
  Lemma~\ref{lem:phi-ast-inv-hoelder} proves the claim with the mean
  value on the left-hand side replaced by $\mean{A(\nabla
    h)}_{2B}$. Due to Lemma~\ref{lem:osc-aux} we can exchange it
  by~$\mean{A(\nabla h)}_B$ which completes the proof.
\end{proof}
Overall, we see that the oscillation of~$A(\nabla h)$ can be measured
with power~$1$ or~$p'$. In particular, if we combine
Theorem~\ref{thm:h-decay} (which still has to be proven) and
Proposition~\ref{pro:A-rev-hoelder}, we get the following decay
estimate in powers of~$p'$.
\begin{proposition}
  \label{pro:h-decay-pprime}
  Let $h\colon \Omega \to \setR$ be $p$-harmonic with $p \geq 2$
  on~$\Omega \subset \setR^2$. Then for all $\beta \in (0,1)$ there
  exists~$c_\beta>0$ such that for all balls~$B \subset \Omega$ and
  all $\theta \in (0,1)$ there holds
  \begin{align*}
    \bigg( \dashint_{\theta B} \abs{A(\nabla h) - \mean{A(\nabla h)}_{\theta
    B}}^{p'} \,dx \bigg)^{\frac 1{p'}}
    &\leq c_\beta\, \theta^\beta
      \dashint_{B} \abs{A(\nabla h) - \mean{A(\nabla h)}_{
      B}} \,dx.
  \end{align*}
\end{proposition}

\subsection{Non-degenerate case}
\label{ssec:non-degenerate-case-h}

Let us consider the non-degenerate case. That is, we will
assume that for some fixed ball $B$ there holds
\begin{align}
  \label{eq:non-deg-cond-h}
  \dashint_B \abs{V(\nabla h) - \mean{V(\nabla h)}_B}^2\,dx &\leq \epsilonDG
  \dashint_B \abs{V(\nabla h)}^2\,dx
\end{align}
with some suitably small~$\epsilonDG>0$.

Inequality \eqref{eq:non-deg-cond-h} means that in this situation~$V(\nabla h)$ (and
therefore also~$\nabla h$) behaves almost like a constant on~$B$. In
particular,~$A(\nabla h) \eqsim \abs{\mean{\nabla h}_B}^{p-2} \nabla
h$. Hence, it is possible to treat the $p$-Laplace equation like a
perturbation of a linear equation with constant coefficients. This
approach was used in~\cite[Proposition~28]{DieLenStrVer12} to prove
(almost) linear decay estimates of the oscillation in this
non-degenerate situation. In fact, it was shown that the decay
estimate for the oscillations of~$V$ even holds in the case of
quasi-convex functionals with Orlicz growth (for any dimension). Our
situation is just a special case. In particular, we obtain:
\begin{lemma}
  \label{lem:V-decay-nondeg}
  For every~$\beta \in (0,1)$ there exists~$c=c(p,\beta)>0$ and
  $\epsilonDG=\epsilonDG(p,\beta)>0$ such that if the non-degeneracy
  condition~\eqref{eq:non-deg-cond-h} holds, then
  \begin{align*}
    \bigg(\dashint_{\theta B} \abs{V(\nabla h) - \mean{V(\nabla
    h)}_{\theta B}}^2\,dx \bigg)^{\frac 12}
    &\leq c\, \theta^\beta \bigg(
      \dashint_{B} \abs{V(\nabla h) - \mean{V(\nabla h)}_B}^2\,dx
      \bigg)^{\frac 12}
  \end{align*}
  for all~$\theta \in (0,1]$.
\end{lemma}
From this (almost) linear decay of the oscillations of~$V$ in the
non-degenerate case, we will now derive (almost) linear decay of the
oscillations of~$A$.
\begin{proposition}
  \label{pro:decay-A-nondeg}
  For every~$\beta \in (0,1)$, there exists~$c=c(p,\beta)>0$ and
  $\epsilonDG=\epsilonDG(p,\beta)>0$ such that if the non-degeneracy
  condition~\eqref{eq:non-deg-cond-h} holds, then
  \begin{align*}
    \dashint_{\theta B}
    \abs{A(\nabla h) - \mean{A(\nabla h)}_{\theta B}}\,dx
    &\leq c\,\theta^{\beta} \dashint_{B}
      \abs{A(\nabla h) - \mean{A(\nabla h)}_B }\,dx
  \end{align*}
  for all~$\theta \in (0,1]$.
\end{proposition}
\begin{proof}
  It suffices to prove the claim for~$\theta = 2^{-m}$
  with~$m \in \setN_0$. For this purpose, let us define $B_m := 2^{-m} B$. 
  Using Lemma~\ref{lem:DKS6.2} and Lemma~\ref{lem:V-decay-nondeg} we can estimate
  \begin{align*}
    I_m &:= \dashint_{B_m} \abs{V(\nabla h) -
                 V(\mean{\nabla h}^A_{B_m})}^2\,dx
    \\
               &\leq c\, \dashint_{B_m} \abs{V(\nabla h) - \mean{V(\nabla h)}_{B_m}}^2\,dx
    \\
               &\leq c\, 2^{-2m\beta} \dashint_{B_1} \abs{V(\nabla h) -
                 \mean{V(\nabla h)}_{B_1}}^2\,dx
    \\
               &\leq c\, 2^{-2m\beta} \dashint_{B_1} \abs{V(\nabla h) -
                 V(\mean{\nabla h}^A_{B_0})}^2\,dx.
  \end{align*}
  The reverse H\"older type estimate in Lemma~\ref{lem:rev-hoelder} with $Q=\mean{\nabla h}_{B_0}^A$ then implies
  \begin{align*}
    I_m  &\leq c\, 2^{-2m\beta} \, \phi^*_{\abs{\mean{A(\nabla h)}_{B_0}}} \bigg( \dashint_B \abs{A(\nabla h) -
                  \mean{A(\nabla h)}_{B_0})}\,dx \bigg).
  \end{align*}
  Using the shift-change Lemma~\ref{lem:shift-change} (with $\lambda=1$) and Lemma~\ref{lem:hammer} we get
  \begin{align*}
    I_m  &\leq c\, 2^{-2m\beta} \Bigg(
                  \phi^*_{\abs{\mean{A(\nabla h)}_{B_m}}} \bigg( \dashint_B \abs{A(\nabla h) -
                  \mean{A(\nabla h)}_{B_0}}\,dx \bigg)
    +              \abs{V(\mean{\nabla h}^A_{B_0}) - V(\mean{\nabla h}^A_{B_m})}^2 \Bigg)
    \\
                &\leq c\, 2^{-2m\beta} \Bigg( \phi^*_{\abs{\mean{A(\nabla h}_{B_m}}} \bigg( \dashint_B \abs{A(\nabla h) -
                  \mean{A(\nabla h)}_{B_0}}\,dx \bigg)
    \\
                &\qquad \qquad \qquad
     + \phi^*_{\abs{\mean{A(\nabla h)}_{B_m}}}(\abs{\mean{A(\nabla h)}_{B_0} -
                  A(\mean{\nabla h}^A_{B_m})}) \Bigg)
    \\
                &\leq c\, 2^{-2m\beta}  \phi^*_{\abs{\mean{A(\nabla
                  h)}_{B_m}}} \bigg( \dashint_B \abs{A(\nabla h) -
                  \mean{A(\nabla h)}_{B_0}}\,dx
             +ot     \abs{\mean{A(\nabla h)}_{B_0} -
                  \mean{A(\nabla h)}_{B_m} } ) \bigg).
  \end{align*}
  Since~$p\geq 2$, we have
  $\phi^*_a(\theta t) \geq c\, \theta^2\, \phi^*_a(t)$ for all
  $a,t \geq 0$ and all~$\theta \in [0,1]$. 
  Thus,
  \begin{align}
    \label{eq:1}
    I_m &\leq c\, \phi^*_{\abs{\mean{A(\nabla h)}_{B_m}}} \Bigg( 2^{-m\beta} \bigg(
                 \dashint_B \abs{A(\nabla h) -
                 \mean{A(\nabla h)}_{B_0}}\,dx  
   \nonumber\\  
    &\qquad \qquad \qquad \qquad\qquad\qquad+
             \abs{\mean{A(\nabla h)}_{B_0} - \mean{A(\nabla h)}_{B_m} } \bigg) \Bigg).
  \end{align}

  On the other hand, with Lemma~\ref{lem:hammer} and Jensen's inequality,
  \begin{align*}
    I_m &= \dashint_{B_m} \abs{V(\nabla h) -
                 V(\mean{\nabla h}^A_{B_m})}^2\,dx
    \\
      &\geq c\, \dashint_{B_m} \phi^*_{\abs{\mean{A(\nabla h)}_{B_m}}}
        \big(\abs{A(\nabla h) - \mean{A(\nabla h)}_{B_m}} \big)\,dx
        \\
      &\geq c\, \phi^*_{\abs{\mean{A(\nabla h)}_{B_m}}} \bigg( \dashint_{B_m}
        \abs{A(\nabla h) - \mean{A(\nabla h)}_{B_m}}\,dx \bigg)
  \end{align*}
  We combine this with~\eqref{eq:1} and apply the inverse
  of~$\phi^*_{\abs{\mean{A(\nabla h)}_{B_m}}}$ to obtain
  \begin{align*}
    \lefteqn{\dashint_{B_m}
    \abs{A(\nabla h) - \mean{A(\nabla h)}_{B_m}}\,dx} \qquad &
                                                               \\
    &\lesssim
      2^{-m\beta} \bigg(
      \dashint_{B_0} \abs{A(\nabla h) -
      \mean{A(\nabla h)}_{B_0}}\,dx  + \abs{\mean{A(\nabla h)}_{B_0} - \mean{A(\nabla h)}_{B_m}} \bigg).
  \end{align*}
  Now,
  \begin{align*}
    \abs{\mean{A(\nabla h)}_{B_0} - \mean{A(\nabla h)}_{B_m}}
    &\leq \sum_{k=0}^{m-1}  \abs{\mean{A(\nabla h)}_{B_k} -
      \mean{A(\nabla h)}_{B_{k+1}}}
    \\
    &\lesssim\sum_{k=0}^{m-1}\dashint_{B_k}
      \abs{A(\nabla h) - \mean{A(\nabla h)}_{B_k}}\,dx
  \end{align*}
  such that from the previous estimate it follows that
  \begin{align*}
    \dashint_{B_m}
    \abs{A(\nabla h) - \mean{A(\nabla h)}_{B_m}}\,dx
    &\leq c\, 2^{-m\beta} \sum_{k=0}^{m-1}\dashint_{B_k}
      \abs{A(\nabla h) - \mean{A(\nabla h)}_{B_k}}\,dx, \quad m\in\setN.
  \end{align*}
  Finally, by Lemma~\ref{lem:iterative} below we conclude
  \begin{align*}
    \dashint_{B_m}
    \abs{A(\nabla h) - \mean{A(\nabla h)}_{B_m}}\,dx
    &\leq c_\beta\,2^{-m\beta} \dashint_{B}
      \abs{A(\nabla h) - \mean{A(\nabla h)}_{B}}\,dx
  \end{align*}
  and the proof is complete.
\end{proof}
In the proof of Proposition \ref{pro:decay-A-nondeg} we have used the following algebraic lemma which is shown here for the sake of completeness.
\begin{lemma}
  \label{lem:iterative}
  Assume that for some~$c_0,\beta>0$ the non-negative sequence $(a_m)_{m\in\setN_0}$ satisfies 
  \begin{align*}
    a_m &\leq c_0 \,2^{-m \beta} \sum_{k=0}^{m-1} a_k, \qquad m\in\setN.
  \end{align*}
  Then for all $m\in\setN$ there holds
  \begin{align*}
    a_m \leq 2^{-m \beta}
    \exp\bigg( \frac{c_0}{1-2^{-\beta}}\bigg) \,a_0.
  \end{align*}
\end{lemma}
\begin{proof}
  Let us define $b_0 := a_0$ and $b_m := c_0\,2^{-m \beta} \sum_{k=0}^{m-1}
  b_k$ for $m\in\setN$. Then by induction there holds~$a_m \leq b_m$. Moreover, it is $b_{m+1} - 2^{-\beta} b_m= c_0 2^{-(m+1)\beta} b_m$ such that
  $b_{m+1} = 2^{-\beta} (1+c_0 2^{-m \beta}) b_m$ and hence
  \begin{align*}
    b_m &= b_0\,2^{-m\beta} \prod_{k=0}^{m-1} \big(1+c_0 2^{-k\beta}), \qquad m\in\setN,
  \end{align*}
  where
  \begin{align*}
  	\prod_{k=0}^{m-1} \big(1+c_0 2^{-k\beta}) 
  	\leq \exp\left( \sum_{k=0}^{\infty} c_0 2^{-k\beta}\right) = \exp\bigg( \frac{c_0}{1-2^{-\beta}}\bigg)=:B.
  \end{align*}
  Thus, we have
  \begin{align*}
	a_m \leq b_m \leq b_0\, 2^{-m\beta} B 
	= 2^{-m \beta} \exp\bigg( \frac{c_0}{1-2^{-\beta}}\bigg) \,a_0,
	\qquad m\in\setN,
\end{align*}  
  as claimed.
\end{proof}

\subsection{Degenerate  case}
\label{ssec:degenerate-case-h}

Let us now turn to the degenerate case.  We need the following
important qualitative regularity result from~\cite{AraTeiUrb17}. Its
proof is based on the estimates for quasi-conformal gradient maps
from~\cite{BaeKov05}.
\begin{lemma}[{\cite{AraTeiUrb17}}]
  \label{lem:teixeira}
  Let $h$ be $p$-harmonic with~$p\geq 2$. Then for
  all~$\theta \in (0, \frac 12]$ it holds
  \begin{align*}
    \sup_{x,y\in\theta B}|\nabla h(x)-\nabla  h(y)|\leq c_0\,
    \theta^{\alpha} \bigg( \dashint_B |\nabla h - \mean{\nabla h}_B|^2 \, dx \bigg)^{\frac 12}.
  \end{align*}
  where
  \begin{align*}
    \alpha = \alpha(p) &= \frac{1}{2p} \bigg( -3 - \frac{1}{p-1} + \sqrt{33 +
                \frac{30}{p-1} + \frac{1}{(p-1)^2}} \bigg) \geq \frac{1}{p-1}.
  \end{align*}
\end{lemma}
\begin{proof}
  We will use the following estimate for complex gradients $\phi$ from~\cite[page~546]{AraTeiUrb17}:
  \begin{align*}
    [\varphi]_{C^{0,\alpha} (B_{\frac 1 2})}
    \le  c(p,\alpha) \, 2^{1+\alpha}\|\nabla \phi\|_{L^2(B_{\frac 1 2})} ,
  \end{align*}
  where~$\alpha=\alpha(p)$ and $c(p,\alpha)=\sqrt{\frac{p-1}{\alpha(p)}}$. In
  our notations it follows that
  \begin{align*}
    \sup_{x,y\in\theta B}|\nabla h(x)-\nabla  h(y)|\leq c_0\,
    \theta^{\alpha} \bigg( \dashint_{\frac 34 B} (R\,|\nabla^2 h|)^2 \, dx
    \bigg)^{\frac 12},
  \end{align*}
  where~$R$ is the radius of~$B$.  Now, the Caccioppoli estimate for
  quasi-conformal maps proves the claim.
\end{proof}
\begin{remark}
  \label{rem:non-optimal-alpha}
  Note that the exponent~$\alpha=\alpha(p)$ from
  Lemma~\ref{lem:teixeira} is smaller than the one
  in~\cite{Aro88,IwaMan89}.  Unfortunately, these articles do not
  provide quantitative estimates, so we have to rely on the possibly
  non-optimal estimate of Lemma~\ref{lem:teixeira}. For example, the
  regularity for~$p$-harmonic maps goes to~$C^{0,1/3}$, as
  $p\to \infty$, but $\lim_{p\to \infty} \alpha(p) =0$.  Nevertheless,
  the exponent from Lemma~\ref{lem:teixeira} is sufficient for (most
  of) our purposes, since $\alpha(p) > \frac 1{p-1}$ for $p>2$. See
  also the discussions in Subsection~\ref{ssec:open-problem}.
\end{remark}

For our purpose we need a $p$-version of the previous lemma.
\begin{lemma}
  \label{lem:teixeira-p}
  Let $h$ be~$p$-harmonic with~$p \geq 2$ let~$\alpha>0$ be as in
  Lemma~\ref{lem:teixeira}. Then for
  all~$\theta \in (0,1/2]$
  \begin{align*}
    \sup_{\theta B}|\nabla h(x)-\nabla h(y)|\leq c_0\, \theta^{\alpha}
    \bigg(\dashint_B |\nabla h - \mean{\nabla h}_B|^p \,
    dx\bigg)^{\frac 1p}.
  \end{align*}
\end{lemma}
\begin{proof}
  This follows from Lemma~\ref{lem:teixeira} if we apply Jensen's
  inequality to the right-hand side using~$p \geq 2$.
\end{proof}

As a further technical step we also need the following (non-optimal) decay
estimate for~$V$ from~\cite{DieStrVer09}.
\begin{lemma}[{\cite[Theorem~6.4]{DieStrVer09}}]
  \label{lem:V-decay}
  There exists~$\gamma>0$ such that for all balls~$B$ and all~$\theta
  \in (0,1)$ there holds
  \begin{align*}
    \bigg(\dashint_{\theta B} \abs{V(\nabla h) - \mean{V(\nabla
    h)}_{\theta B}}^2\,dx \bigg)^{\frac 12}
    &\leq c\, \theta^{\gamma} \bigg(
      \dashint_{B} \abs{V(\nabla h) - \mean{V(\nabla h)}_B}^2\,dx
      \bigg)^{\frac 12}.
  \end{align*}
\end{lemma}
\begin{remark}
  \label{rem:V-decay1}
  The decay estimate in Lemma~\ref{lem:V-decay} is proven for any
  dimension. However, it provides no explicit lower bound
  for~$\gamma>0$. Therefore, it only provides a very slow, non-optimal
  decay for~$V$. See below in Subsection~\ref{ssec:open-problem} for
  discussions.
\end{remark}
Now we have enough tools at hand to prove an important assertion on alternatives:
\begin{proposition}
  \label{pro:deg-alternatives}
  Let~$h$ be $p$-harmonic with~$p \geq 2$ and let~$\beta \in
  (0,1)$. Suppose that~$h$ fails the non-degeneracy
  condition~\eqref{eq:non-deg-cond-h}, i.e., we have
  \begin{align*}
    \dashint_B \abs{V(\nabla h) - \mean{V(\nabla h)}_B}^2\,dx
    &> \epsilonDG
      \dashint_B \abs{V(\nabla h)}^2\,dx.
  \end{align*}
  Then there exist~$\theta_2=\theta_2(p,\beta, \epsilonDG)\in(0, \frac 12)$ and $m\geq 2$ such that for $\theta_1:=\theta_2^m$ 
  at least one of the following alternative applies:
  \begin{enumerate}
  \item \label{itm:deg-alt1} $h$ satisfies the non-degeneracy
    condition~\eqref{eq:non-deg-cond-h} on the ball~$\theta_1 B$.
  \item \label{itm:deg-alt2} $h$ satisfies the decay estimate
    \begin{align*}
      \bigg(\dashint_{\theta_2 B} \abs{A(\nabla h) - \mean{A(\nabla
      h)}_{\theta_2 B}}^{p'}\,dx \bigg)^{\frac{1}{p'}}
      &\leq \theta_2^{\beta}
        \bigg(\dashint_{B} \abs{A(\nabla h) - \mean{A(\nabla
        h)}_{B}}^{p'}\,dx\bigg)^{\frac{1}{p'}}.
    \end{align*}
  \end{enumerate}
\end{proposition}
\begin{proof}
  Without loss of generality we can assume that~$0$ is the center of~$B$.
  Suppose that for $\theta_1$ (to be specified later) alternative~\ref{itm:deg-alt1} fails on~$\theta_1 B$, i.e., that
  \begin{align*}
    \dashint_{\theta_1 B} \abs{V(\nabla h)}^2\,dx
    &< \frac{1}{\epsilonDG}
      \dashint_{\theta_1 B} \abs{V(\nabla h) - \mean{V(\nabla h)}_{\theta_1 B}}^2\,dx.
  \end{align*}
  Then the (non-optimal) decay estimate of~$p$-harmonic
  functions of Lemma~\ref{lem:V-decay} implies that there
  exists~$\gamma>0$ such that
  \begin{align*}
    \dashint_{\theta_1 B} \abs{V(\nabla h) - \mean{V(\nabla
    h)}_{\theta_1 B}}^2\,dx
    &\lesssim \theta_1^{2\gamma}
    \dashint_B \abs{V(\nabla h) - \mean{V(\nabla
    h)}_B}^2\,dx.
  \end{align*}
  So, the~$L^\infty$-estimate from~\cite[Lemma 5.8]{DieStrVer09} together with the 
  two previous bounds implies
  \begin{align}
    \label{eq:2}
    \abs{\nabla h(0)}^p &\leq c\,     \dashint_{\theta_1 B} \abs{V(\nabla h)}^2\,dx
       \lesssim \frac{\theta_1^{2\gamma}}{\epsilonDG}
                        \dashint_B \abs{V(\nabla h) - \mean{V(\nabla
                        h)}_B}^2\,dx
  \end{align}
  Moreover, for the larger ball~$\theta_2 B$ we employ Lemma \ref{lem:osc-aux} to derive
  \begin{align*}
    \dashint_{\theta_2 B} \abs{A(\nabla h) - \mean{A(\nabla
    h)}_{\theta_2 B}}^{p'} \,dx
    &\lesssim
      \dashint_{\theta_2 B} \abs{A(\nabla h)}^{p'} \,dx
    \\
    &=   \dashint_{\theta_2 B} \abs{\nabla h}^{p} \,dx
    \\
    &\lesssim \sup_{x\in \theta_2 B} (\abs{\nabla h(x)- \nabla h(0)}\big)^{p}  
    + \abs{\nabla h(0)}^p.
  \end{align*}
  This,~Lemma~\ref{lem:teixeira-p} (using~$\abs{V(\nabla h)}^2 =
  \abs{\nabla h}^p$), and~\eqref{eq:2} imply that
  \begin{align*}
    &\dashint_{\theta_1 B} \abs{A(\nabla h) - \mean{A(\nabla
    h)}_{\theta_2 B}}^{p'} \,dx 
   \lesssim \theta_2^{\alpha p} \dashint_B \abs{V(\nabla h)}^2\,dx +
      \frac{\theta_1^{2\gamma}}{\epsilonDG}
                        \dashint_B \abs{V(\nabla h) - \mean{V(\nabla
                        h)}_B}^2\,dx.
  \end{align*}
  Since~$h$ fails the non-degeneracy
condition~\eqref{eq:non-deg-cond-h} on~$B$, we obtain
  \begin{align*}
    \dashint_{\theta_2 B} \abs{A(\nabla h) - \mean{A(\nabla
    h)}_{\theta_2 B}}^{p'} \,dx
    &\lesssim \bigg( \frac{\theta_2^{\alpha p}}{\epsilonDG} +
      \frac{\theta_1^{2\gamma}}{\epsilonDG} \bigg)
      \dashint_B \abs{V(\nabla h) - \mean{V(\nabla
      h)}_B}^2\,dx \\
    &\leq \frac{c}{\epsilonDG} \, \big( \theta_2^{\alpha p}  +\theta_1^{2\gamma} \big)
      \dashint_B \abs{A(\nabla h) - \mean{A(\nabla
  		h)}_B}^{p'}\,dx,
  \end{align*}
  where for the second estimate we used that $\abs{V(P) - V(Q)}^2 \lesssim \abs{A(P)-A(Q)}^{p'}$ according to
  Corollary \ref{cor:hammer-pge2} and $p\geq 2$.

  Let us assume in the following that~$p>2$, since the claim of the
  lemma is standard for the linear case~$p=2$. Hence, $\alpha$ from Lemma \ref{lem:teixeira-p} satisfies $\alpha > \frac{1}{p-1}$ such that $\alpha p > p' > \beta p'$. Therefore we can
  choose~$\theta_2$ so small
  that~$c\,\theta_2^{\alpha p} / \epsilonDG \leq \frac 12
  \theta_2^{\beta p'}$.
  Now choose $m\geq 2$ large enough such that $\theta_1:=\theta_2^m$ satisfies 
 $c\,\theta_1^{2\gamma} /\epsilonDG \leq  \frac 12 \theta_2^{\beta p'}$. So, we finally obtain
  \begin{align*}
    \dashint_{\theta_2 B} \abs{A(\nabla h) - \mean{A(\nabla
    h)}_{\theta_2 B}}^{p'} \,dx
    &\leq \theta_2^{\beta p'}
      \dashint_B \abs{A(\nabla h) - \mean{A(\nabla
      h)}_B}^{p'}\,dx.
  \end{align*}
  This proves the claim.
\end{proof}

\subsection{Proof of Theorem~\ref{thm:h-decay}}
\label{ssec:proof-theor-refthm:h}

We are now prepared to prove our 
decay estimates for~$A(\nabla
h)$.
\begin{proof}[Proof of Theorem~\ref{thm:h-decay}]
  Given $p\geq 2$ and $\beta \in (0,1)$ fix~$\epsilonDG$ such that
  Proposition~\ref{pro:decay-A-nondeg} is applicable and
  choose~$\theta_2 \in (0, \tfrac 12)$ and $m\geq 2$ according to
  Proposition~\ref{pro:deg-alternatives}. Then it is enough to prove the
  claim for the special sequence of balls $B_k := \theta_2^k B$, $k\in\setN_0$. It indeed suffices to show that 
  \begin{align*}
	  \dashint_{B_k} \abs{A(\nabla h) - \mean{A(\nabla h)}_{B_k}}\,dx
 	  \leq c\,\theta_2^{k\beta} \dashint_{B_0} \abs{A(\nabla h) - \mean{A(\nabla h)}_{B_0}} \,dx, \quad k\in\setN\setminus\{1\}.
  \end{align*}
  To this end, let $k_0\in\setN_0$ denote the smallest number $k$ such that $h$ satisfies the non-degeneracy condition~\eqref{eq:non-deg-cond-h} on $B_k$ or on $B_{k+m}$. 
  If no such $k$ exists, we set $k_0:=\infty$.
  Then for every $k\in \setN$ with $k<k_0$ condition~\eqref{eq:non-deg-cond-h} is violated for $B_k$ and $\theta_1 B_{k}=B_{k+m}$. Therefore, the second alternative of Proposition~\ref{pro:deg-alternatives} applies and we inductively conclude that
  \begin{align*}
    \bigg( \dashint_{B_{k+1}} \abs{A(\nabla h) - \mean{A(\nabla
    h)}_{B_{k+1}}}^{p'}\,dx
    \bigg)^{\frac 1{p'}} \leq \theta_2^{k\beta}
    \bigg( \dashint_{B_1} \abs{A(\nabla h) - \mean{A(\nabla
    h)}_{B_1}}^{p'}\,dx \bigg)^{\frac 1{p'}}
  \end{align*}
  for all $1\leq k < k_0$.
    Using Jensen's inequality on the left-hand side and
  Proposition~\ref{pro:A-rev-hoelder} for the right-hand side (note that~$2B_1 \subset B_0$, since $\theta_2 < \frac 12$), we conclude the desired estimate for all $2 \leq k < k_0+1$.
    
    If $k_0=\infty$, the proof is finished. Otherwise, if $k_0\in\setN_0$, we are left with showing that for all $k>k_0$ there holds 
  \begin{align}\label{eq:k0}
	  \dashint_{B_k} \abs{A(\nabla h) - \mean{A(\nabla h)}_{B_k}}\,dx
 	  \leq c\,\theta_2^{(k-k_0)\beta} \dashint_{B_{k_0}} \abs{A(\nabla h) - \mean{A(\nabla h)}_{B_{k_0}}} \,dx.
  \end{align}
  By construction of~$k_0$, our solution satisfies the non-degeneracy condition~\eqref{eq:non-deg-cond-h} on $B_{k_0}$ or $B_{k_0+m}$. 
  In the first case, \eqref{eq:k0} directly follows from Proposition~\ref{pro:decay-A-nondeg} and the proof is complete. For the the second case, the same assertion yields that
  \begin{align*}
	  \dashint_{B_k} \abs{A(\nabla h) - \mean{A(\nabla h)}_{B_k}}\,dx
 	  \leq c\,\theta_2^{(k-(k_0+m))\beta} \dashint_{B_{k_0+m}} \abs{A(\nabla h) - \mean{A(\nabla h)}_{B_{k_0+m}}} \,dx
  \end{align*}
  for every $k \geq k_0+m$. Finally, it remains to note that for each $\ell\in\{0,1,\ldots,m\}$ we may estimate
  \begin{align*}
    \dashint_{B_{k_0+\ell}} \abs{A(\nabla h) - \mean{A(\nabla
    h)}_{B_{k+\ell}}}\,dx
    \leq c\,\theta_2^{\ell\beta}
    \dashint_{B_{k_0}} \abs{A(\nabla h) - \mean{A(\nabla
    h)}_{B_{k_0}}} \,dx,
  \end{align*}  
  since $\theta_2$ is assumed to be fixed. The combination of the last two bounds then shows \eqref{eq:k0} which completes the proof of Theorem~\ref{thm:h-decay}.
\end{proof}

\subsection{\texorpdfstring{The case $1 < p \leq 2$ and proof of
    Theorem~\ref{thm:h-decay-grad}}{The case 1 < p <= 2 and proof of
    Theorem \ref{thm:h-decay-grad}}}
\label{ssec:duality}

The situation for~$1<p<2$ strongly differs from the case~$p \geq
2$. Let us explain in this subsection what kind of results can be obtained
in this situation.

So let us assume here that~$h$ is~$p$-harmonic
on $\Omega\subset \setR^2$ with~$1<p<2$. The optimal regularity of such functions has been studied in detail in~\cite{Aro88,IwaMan89}. In particular, it has been shown
that~$\nabla h \in C^{k,\widetilde{\alpha}}_{\loc}(\Omega)$ with
\begin{align*}
  k+\widetilde{\alpha} & = \frac 16 \bigg( 1 + \frac{1}{p-1} + \sqrt{1 +
             \frac{14}{p-1} + \frac{1}{(p-1)^2}} \bigg)=: \eta(p).
\end{align*}
The optimality of this regularity result has been shown already by
Dobrowolski~\cite[Remark in Section~2]{Dob83}.  Expressed
in~$A(\nabla h)$ this translates to
$A(\nabla h) \in C^{\ell,\widetilde{\beta}}_{\loc}(\Omega)$ with
\begin{align*}
  \ell+\widetilde{\beta} &= \eta(p) (p-1) = \eta(p').
\end{align*}
Note that $\ell+\widetilde{\beta} \in (0,1)$ for~$p <2$ and
$\eta(p) \searrow \frac 13$, as $p \to \infty$. 
In particular, in general $A(\nabla h) \notin C^1$ for $p < 2$. By the
same argument it follows that for~$p<2$ we have
$A(\nabla h) \notin \bfB^s_{\rho,q}$ and
$A(\nabla h) \notin \bfF^s_{\rho,q}$ for any~$s \in (0,1)$ such that
$s - \frac{2}{\rho} > \ell + \widetilde{\beta}$. Therefore, it is not
possible to obtain (almost) linear decay estimates of~$A(\nabla h)$ as
in Theorem~\ref{thm:h-decay}.  Moreover,
Theorem~\ref{thm:reg-transfer} below cannot hold in full generality
for~$p < 2$, since it fails already for~$F=0$.

The natural object to look at for~$1<p\leq 2$ is~$\nabla u$ rather
than~$A(\nabla u)$. This becomes more clear by duality in the language
of differential forms. Indeed, we can use the following nice duality
trick from Hamburger~\cite[Section~5]{Ham92}. Let us assume that~$h$ is
$p$-harmonic and let us interpret it as a $0$-form. Then the $1$-form
$\omega := dh$ (which corresponds to~$\nabla h$) satisfies
\begin{align*}
  \delta A(\omega)=0, \qquad d\omega = 0.
\end{align*}
Now, if we define the $1$-form~$\tau$ by $*\tau := A(\omega)$ (using
that we are in two space dimensions), then
\begin{align*}
  \delta A^{-1}(\tau) =0, \qquad d \tau = 0.
\end{align*}
Since~$d \tau =0$, we find a $0$-form~$z$ with $\tau = dz$. Due to
$A^{-1}(Q) = \abs{Q}^{p'-2} Q$, we obtain that~$z$ is
$p'$-harmonic. In particular,~$h$ is $p$-harmonic if and only if its
conjugate solution~$z$ is $p'$-harmonic. Moreover, we have the
relation~$*\tau = A(\omega)$ and thus $\omega = A^{-1}(*\tau)$. 
This allows
to transfer estimates from~$A(\omega)$ to~$\tau$ and from $A^{-1}(*\tau)$
to~$\omega$. 

\begin{proof}[Proof of Theorem~\ref{thm:h-decay-grad}]
  If $h$ is $p$-harmonic with~$p \in (1,2)$, then its conjugate~$z$ is
  $p'$-harmonic with $p'>2$. Hence, from Theorem~\ref{thm:h-decay} we
  obtain decay estimates for the oscillations
  of~$A^{-1}(\tau)$. Using~$\omega=A^{-1}(*\tau)= *A^{-1}(\tau)$ this
  directly implies decay estimates for~$\omega$, resp. $\nabla u$. In
  particular, this proves Theorem~\ref{thm:h-decay-grad}.
\end{proof}

\subsection{Open problems}
\label{ssec:open-problem}
We have established in Theorem \ref{thm:h-decay} an (almost) linear
decay of the oscillation of~$A(\nabla h)$. This is optimal in the
sense that oscillations can never decay faster than linear. However, the limiting case of linear decay unfortunately is excluded by our method of proof. So we ask the following question:
\begin{question}
  Is it possible to obtain linear decay of the $A(\nabla h)$-oscillations for
  $p\geq 2$, i.e., does Theorem~\ref{thm:h-decay} also hold with~$\beta=1$?
\end{question}

\begin{remark}
  Let us remark that due to the
  behavior of $p$-harmonic functions $h$ in the plane a linear decay of $A(\nabla h)$-oscillations is only possible for~$p \geq 2$, see
  Subsection~\ref{ssec:duality}. In the case~$p \leq 2$, the corresponding question
  would be to obtain linear decay of $\nabla h$-oscillations, i.e., Theorem~\ref{thm:h-decay-grad} with~$\beta=1$.
\end{remark}

Parts of the proofs in this section and Section~\ref{sec:oscill-estim}
are based on the decay of $V(\nabla h)$-oscillations. However, the
decay estimate in Lemma~\ref{lem:V-decay} is non-optimal in the sense that it
provides no sharp lower bound for the decay exponent~$\gamma>0$. We
have used the estimate of~\cite{AraTeiUrb17} in order to prove an
(almost) optimal decay of the $A(\nabla h)$-oscillations. The same
method can be used to prove decay estimates for~$V(\nabla h)$ (in the
plane for $p \geq 2$). With the exponent~$\alpha=\alpha(p)>0$ from
Lemma~\ref{lem:teixeira} we obtain the following estimate:
\begin{lemma}
  Let $p \geq 2$ and let $h$ be a $p$-harmonic function in the
  plane. Then for every $\beta \in (0,\frac{\alpha\, p}{2})$ there
  exists~$c>0$ such that
  \begin{align*}
    \bigg(\dashint_{\theta B} \abs{V(\nabla h) - \mean{V(\nabla
    h)}_{\theta B}}^2\,dx \bigg)^{\frac 12}
    &\leq c\, \theta^{\beta} \bigg(
      \dashint_{B} \abs{V(\nabla h) - \mean{V(\nabla h)}_B}^2\,dx
      \bigg)^{\frac 12}
  \end{align*}
  for all~$\theta \in (0,1]$.
\end{lemma}
Note that $\frac{\alpha\, p}{2} < 1$ for all $p >2$. Moreover, the quantity $\frac{\alpha\, p}{2}$ is decreasing in~$p$ and
$\lim_{p \to \infty} \frac{\alpha\, p}{2} = \frac{\sqrt{33}-3}{4}
\approx 0.6861$. It is already mentioned in~\cite{AraTeiUrb17} that it
is possible to improve~$\alpha(p)$ form Lemma~\ref{lem:teixeira} a
tiny bit by using Young's inequality in the proof. However, this
improvement is not enough to raise $\frac{\alpha\, p}{2}$ above~$1$ and
the limit at~$p=\infty$ stays the same. So the $V$-decay is still
strongly sub-linear. However, note that $\alpha > \tfrac{1}{p-1}$ for $p >2$, which is
the reason why we can still derive (almost) optimal decay for
$A(\nabla h)$-oscillations.

Nevertheless, the regularity studies in \cite{Aro88,IwaMan89} of
$p$-harmonic functions in the plane indicate a better regularity
of~$V(\nabla h)$, which would allow a linear decay of the
$V(\nabla h)$-oscillations
for all $1<p<\infty$. This is in contrast to the regularity
of~$\nabla h$ and $A(\nabla h)$. In particular, $\nabla h \in C^1$ is
only possible for $p \leq 2$ and $A(\nabla h) \in C^1$ is only possible
for~$p \geq 2$. However, it seems that $V(\nabla h) \in C^1$ for
all~$p$. We strongly believe that this is
also the natural regularity for higher dimensions and vectorial
solutions. Therefore we raise the following conjecture:
\begin{conjecture}
  \label{conj:VC1}
  For $d,n\in\setN$ and $1 < p < \infty$ let $h \colon \Omega \to \setR^n$ be
  $p$-harmonic on $\Omega \subset \Rd$. Then
  $V(\nabla h) \in C^1(\Omega)$ and we have a linear decay, i.e.,
  \begin{align*}
    \bigg(\dashint_{\theta B} \abs{V(\nabla h) - \mean{V(\nabla
    h)}_{\theta B}}^2\,dx \bigg)^{\frac 12}
    &\leq c\, \theta \bigg(
      \dashint_{B} \abs{V(\nabla h) - \mean{V(\nabla h)}_B}^2\,dx
      \bigg)^{\frac 12}
  \end{align*}
  for all balls $B\subset\Omega$ and every~$\theta \in (0,1]$.
\end{conjecture}
Note that $V(\nabla h) \in C^1$ immediately implies that
$A(\nabla h) \in C^1$ for $p \geq 2$ and $\nabla h \in C^1$ for
$p \leq 2$. In particular, for $p \geq 2$ it follows that
$\nabla h \in C^{1,\frac{1}{p-1}}$ and therefore $h \in C^{p'}$ (in
the sense of H\"older spaces). Thus the
conjecture is stronger than the well known $p'$-conjecture,
see~\cite{AraTeiUrb17}. In addition, an almost linear decay of the $V(\nabla h)$-oscillations would
simplify a few steps in Section~\ref{sec:oscill-estim} below.

\section{Oscillation Estimates}
\label{sec:oscill-estim}
In this section we will derive decay estimates for oscillations of the
flux~$A(\nabla u)$. These will be crucial later in deriving
\Calderon{}-Zygmund type estimates for~$A(\nabla u)$ in the scale of
Besov or Triebel-Lizorkin spaces. The goal of this section is the
proof of the following estimate.

\begin{theorem}
  \label{thm:osc-estimates}
  Let~$2 \leq p < \infty$ and $\Omega\subset\setR^2$. For given $F \in L^{p'}(\Omega)$
  let $u \in W^{1,p}(\Omega)$ be a (scalar) weak solution to
  \begin{align}\label{eq:plap2}
    -\divergence(A(\nabla u)) &= -\divergence F \qquad \text{in $\Omega$}.
  \end{align}
  Then for
  all $\beta \in (0,1)$ there exists~$\theta_0 \in (0,1)$
  and~$c=c(\beta,\theta_0)>0$ such that for all balls~$B \subset \Omega$ there holds
  \begin{align*}
    &\bigg(\dashint_{\theta_0 B} \abs{A(\nabla u) - \mean{A(\nabla u)}_{\theta_0 B}}^{p'} \,dx \bigg)^{\frac 1{p'}} \\
     &\qquad \leq \theta_0^\beta \bigg(
     \dashint_{B} \abs{A(\nabla u) - \mean{A(\nabla u)}_{
       B}}^{p'} \,dx \bigg)^{\frac 1{p'}}
 + c\, \bigg(
     \dashint_{B} \abs{F - \mean{F}_{
       B}}^{p'} \,dx \bigg)^{\frac 1{p'}}.
  \end{align*}
\end{theorem}
\begin{remark}
  \label{rem:osc-est-Rd}
  In the case of higher dimensions and vectorial solutions we get the
  same oscillation decay estimate but with~$\beta$ restricted
  to~$\beta \in (0, \beta_0)$, where~$\beta_0 \in(0,1)$ is some (unknown) small number. The reason is the worse decay estimate for
  $p$-harmonic functions in this more general situation, see
  Remark~\ref{rem:h-decay-Rd}. In fact, our oscillation estimates hold
  exactly in the same range as the decay estimate for $p$-harmonic
  functions.
\end{remark}
Theorem~\ref{thm:osc-estimates} shows that the oscillation
of~$A(\nabla u)$ decreases for some fixed(!) reduction of the
radius by a factor of~$\theta_0$. We can iterate the estimate to obtain an
oscillation decay for arbitrary reductions~$\theta \in
(0,1)$. However, to formulate this it is useful to introduce the
following short notations on oscillations.

Let $B_t(x)$ denote the ball of radius $t>0$ centered at $x$. Then for
$g \in L^w_{\loc}(\Rd)$, $w\geq 1$, we define its (zero order)
oscillation by
\begin{align*}
  \osc_wg(x,t) &:= \bigg( \dashint_{B_t(x)} \abs{g(y) - \mean{g}_{B_t(x)}}^w \,dy
                \bigg)^{\frac 1w}.
\end{align*}
Note that in this definition it is possible to replace the mean by an infimum over all constants, which gives rise to an equivalent expression, see Lemma~\ref{lem:osc-aux}.

\begin{theorem}
  \label{thm:osc-estimates-general}
  Let $u$, $p$, $F$, and $\beta$ be as in
  Theorem~\ref{thm:osc-estimates}. Then there exists~$c>0$ such that
  for all~$\theta \in (0,1)$ and all balls~$B_t(x) \subset \Omega$
  there holds
  \begin{align*}
    \osc_{p'} A(\nabla u)(x,\theta t) \leq c\,\theta^\beta
    \osc_{p'} A(\nabla u)(x,t) + c\, \theta^\beta \int_{\theta}^1
    \lambda^{-\beta}  \osc_{p'} F(x, \lambda t) \frac{d\lambda}{\lambda}.
  \end{align*}
\end{theorem}
Both theorems will be proven in
Subsection~\ref{ssec:proof-theorem-osc-est}. 

\subsection{Non-linear comparison}
\label{ssec:non-line-comp}

In the proof of our oscillation estimates, we will need to compare the
function~$u$ locally with the $p$-harmonic function~$h$ that solves
\begin{align}
  \label{eq:nolin-comparison}
  \begin{alignedat}{2}
    -\divergence\big(A(\nabla h)\big) &= 0 &\qquad&\text{in~$B$},
    \\
    h &= u &&\text{on~$\partial B$}.
  \end{alignedat}
\end{align}
The basic idea is to transfer the decay estimate of~$V(\nabla h)$
 to~$V(\nabla u)$, resp.\ $A(\nabla h)$ to $A(\nabla u)$, by using
the following comparison result.
\begin{lemma}[Non-linear comparison]
  \label{lem:nonlin-comparison-V}
  Let~$h$ be the solution of~\eqref{eq:nolin-comparison}. Then
  \begin{align*}
    \dashint_B \abs{V(\nabla u) - V(\nabla h)}^2\,dx
    &\lesssim 
      \dashint_B \phi^*_{\abs{A(\nabla u)}}(\abs{F-\mean{F}_B})\,dx.
  \end{align*}
\end{lemma}
\begin{proof}
  We take take the difference of the equations for~$u$ and~$h$ and test
  it with $u-h$ scaled by~$\abs{B}^{-1}$. So for arbitrarily small $\delta>0$ we obtain 
  \begin{align*}
	\int_B \abs{V(\nabla u) - V(\nabla h)}^2\,dx 
	& \lesssim \int_B \big( A(\nabla u) - A(\nabla h) \big) \cdot \big( \nabla u - \nabla h\big) dx \\	
	& \leq \int_B \abs{ F-\mean{F}_B}\, \abs{\nabla u - \nabla h} dx \\
	& \leq c_\delta\,\int_B \phi^*_{\abs{A(\nabla
		u)}}(\abs{F-\mean{F}_B})\,dx + \delta\,
\int_B \phi_{\abs{\nabla u}}(\abs{\nabla u - \nabla h})\,dx
\\
	& \lesssim c_\delta\,\int_B \phi^*_{\abs{A(\nabla
		u)}}(\abs{F-\mean{F}_B})\,dx + \delta\,
\int_B \abs{V(\nabla u) - V(\nabla h)}^2\,dx,
\end{align*}
  where we have used Lemma~\ref{lem:hammer}, as well as
  $(\phi_{\abs{\nabla u}})^* \eqsim \phi^*_{\abs{A(\nabla u)}}$. 
  Now we absorb the last integral to prove the
  claim.
\end{proof}

In the following we need an estimate of reverse H\"older
type from~\cite[Corollary~2.4]{BreCiaDieKuuSch17} which is also
contained in the proof of~\cite[Corollary~3.5]{DieKapSch11}.
\begin{lemma}[{\cite[Corollary~2.4]{BreCiaDieKuuSch17}}]
  \label{lem:reverse-hoelder}
  Let $u$ solve~\eqref{eq:plap2}. Then for all vectors $Q$ we have
  \begin{align*}
    &\dashint_{B} \abs{V(\nabla u) - V(Q)}^2\,dx 
    \\
    &\qquad 
   \leq \phi^*_{\abs{A(Q)}} \bigg( \dashint_{2B} \abs{A(\nabla u) - A(Q)}\,dx \bigg)
    + c\, \dashint_{2B}
      \phi^*_{\abs{A(Q)}}(\abs{F - \mean{F}_{2B}})\,dx.
  \end{align*}
\end{lemma}

Now the decay assertion for~$V(\nabla h)$ in Lemma~\ref{lem:V-decay} provides us with a preliminary decay estimate for~$V(\nabla u)$. However note that the decay exponent is far from being optimal. Anyhow, we need this decay estimate to control our final oscillation on a small subset.
\begin{lemma}
  \label{lem:decay-u-non-deg-V}
  Let~$\gamma>0$ be as in Lemma~\ref{lem:V-decay}. Then there exists $c=c(\gamma)>0$ such that we have the following decay
  estimate:
  \begin{align*}
    \lefteqn{\dashint_{\theta B}  \abs{V(\nabla u) - \mean{V(\nabla u)}_{\theta
    B}}^{2}\,dx} \qquad
    &
    \\
    &\leq       c\, \theta^{2\gamma} \dashint_{2B}
      \abs{A(\nabla u)- \mean{A(\nabla u)}_{2B}}^{p'} \,dx +
      c\,\theta^{-d}
      \dashint_{2B} \abs{F-\mean{F}_{2B}}^{p'} \,dx.
  \end{align*}
\end{lemma}
\begin{proof}
  Let $h$ be the solution of~\eqref{eq:nolin-comparison}. We estimate
  \begin{align*}
    I_\theta 
    &:= \dashint_{\theta B}  \abs{V(\nabla u) - \mean{V(\nabla u)}_{\theta B}}^2\,dx  \\
    &\lesssim
      \dashint_{\theta B}  \abs{V(\nabla h) - \mean{V(\nabla h)}_{\theta
      B}}^2 \,dx +
      \dashint_{\theta B}  \abs{V(\nabla u) - V(\nabla h)}^2\,dx
  \end{align*}
  and use the decay estimate for~$V(\nabla h)$, see
  Lemma~\ref{lem:V-decay}, together with $\theta^{2\gamma}<\theta^{-d}$ to conclude that
  \begin{align*}
    I_\theta &\lesssim
        \theta^{2\gamma} \dashint_{B} \abs{V(\nabla h) - \mean{V(\nabla h)}_{B}}^{2}\,dx 
        + \theta^{-d}
        \dashint_{B}  \abs{V(\nabla u) - V(\nabla h)}^{2}\,dx
        \\
    &\lesssim \theta^{2\gamma} \dashint_{B}  \abs{V(\nabla u) - \mean{V(\nabla u)}_{
        B}}^{2}\,dx + \theta^{-d}
      \dashint_{B}  \abs{V(\nabla u) - V(\nabla h)}^{2}\,dx.
  \end{align*}  
  Now, Lemma~\ref{lem:DKS6.2} and Lemma~\ref{lem:nonlin-comparison-V}, imply
  \begin{align*}
	I_\theta 
		&\lesssim \theta^{2\gamma} \dashint_{B} \abs{V(\nabla u) - V \!\left(\mean{\nabla u}_{B}^A \right) }^{2}\,dx + \theta^{-d}
\dashint_B \phi^*_{\abs{A(\nabla u)}}(\abs{F-\mean{F}_{B}})\,dx.
  \end{align*}
  For the first integral we can employ Lemma~\ref{lem:reverse-hoelder} with $Q:=\mean{\nabla u}_{B}^A$ and Corollary~\ref{cor:hammer-pge2} to obtain
  \begin{align*}
	 &\dashint_{B} \abs{V(\nabla u) - V \!\left(\mean{\nabla u}_{B}^A\right)}^{2}\,dx
	  \\ 
	 &\qquad \leq \phi^*_{\abs{ A(Q) }} \bigg( \dashint_{2B} \abs{A(\nabla u)-\mean{A(\nabla u)}_B }\,dx \bigg) 
	 + c \dashint_{2B} \phi^*_{\abs{ A(Q)}}(\abs{F - \mean{F}_{2B}})\,dx \\
	 &\qquad \lesssim \dashint_{2B} \abs{A(\nabla u)-\mean{A(\nabla u)}_B }^{p'}\,dx  
	 + \dashint_{2B} \abs{F - \mean{F}_{2B}}^{p'} \,dx \\
	 &\qquad \lesssim \dashint_{2B} \abs{A(\nabla u)-\mean{A(\nabla u)}_{2B} }^{p'}\,dx + \theta^{-2\gamma-d} \dashint_{2B} \abs{F - \mean{F}_{2B}}^{p'} \,dx.
  \end{align*}
  Similarly the second integral can be estimated by
  \begin{align*}
	\dashint_B \phi^*_{\abs{A(\nabla u)}}(\abs{F-\mean{F}_{B}})\,dx
	&\lesssim \dashint_B \abs{F-\mean{F}_{B}}^{p'} \,dx
	\lesssim \dashint_{2B} \abs{F-\mean{F}_{2B}}^{p'} \,dx.
  \end{align*}
  Hence, combining the last two bounds shows the claim.
\end{proof}

\subsection{Degenerate case}
\label{ssec:degenerate-case-u}

Let us begin with the degenerate case. In particular, we assume
that
\begin{align}
  \label{eq:deg-cond-u}
  \dashint_B \abs{V(\nabla u) - \mean{V(\nabla u)}_B}^2\,dx &> \epsilonDG
  \dashint_B \abs{V(\nabla u)}^2\,dx.
\end{align}
The parameter~$\epsilonDG>0$ is fixed in this section. The specific
value of~$\epsilonDG$ will be determined later by the non-degenerate
case.

We are now prepared to prove the desired $A$-decay estimate.
\begin{proposition}
  \label{pro:decay-u-deg}
  Let $\beta\in (0,1)$. Then there exists a constant~$c=c(\beta)>0$ such that for
  every~$\theta,\epsilonDG \in (0,1)$ we have the following decay
  estimate on balls $B$ with~\eqref{eq:deg-cond-u}:
  \begin{align*}
    \bigg(\dashint_{\theta B}  \abs{A(\nabla u) - \mean{A(\nabla u)}_{\theta B}}^{p'}\,dx\bigg)^{\frac 1{p'}}    
    &\leq c\, \theta^\beta \bigg(\dashint_{2B}
      \abs{A(\nabla u)- \mean{A(\nabla u)}_{2B}}^{p'}\,dx\bigg)^{\frac 1{p'}} \\
    &\quad\qquad + c\,
      \epsilonDG^{-\frac{(p-1)^2}{p}}  \theta^{-(p-1)(\beta+d)}
      \bigg( \dashint_{2B} \abs{F-\mean{F}_{2B}}^{p'}\,dx \bigg)^{\frac 1{p'}}.
  \end{align*}
\end{proposition}
\begin{proof}
  Let $h$ be the solution of~\eqref{eq:nolin-comparison}. Similar to the proof of Lemma~\ref{lem:decay-u-non-deg-V} we estimate
  \begin{align*}
    I_\theta &:= \dashint_{\theta B}  \abs{A(\nabla u) - \mean{A(\nabla u)}_{\theta
    B}}^{p'}\,dx \\
    &\lesssim \dashint_{\theta B}  \abs{A(\nabla h) - \mean{A(\nabla h)}_{\theta
      B}}^{p'}\,dx 
  + \dashint_{\theta B}  \abs{A(\nabla u) - A(\nabla h)}^{p'}\,dx,
  \end{align*}
  where this time the decay estimate for~$A(\nabla h)$, see Proposition~\ref{pro:h-decay-pprime}, implies
  \begin{align*}
    I_\theta &\lesssim \theta^{p'\beta} \dashint_{B}  \abs{A(\nabla h) - \mean{A(\nabla h)}_{B}}^{p'}\,dx + \theta^{-d} \dashint_{B}  \abs{A(\nabla u) - A(\nabla h)}^{p'}\,dx.
  \end{align*}
  Now we shall show that the second integral is bounded by
  \begin{align*}
  R &:= \dashint_{B}  \abs{A(\nabla u) - A(\nabla h)}^{p'}\,dx \nonumber\\
  &\lesssim \theta^{p'\beta+d} \dashint_B \abs{A(\nabla u)- \mean{A(\nabla u)}_B}^{p'}\,dx
  + \epsilonDG^{1-p}\,  \theta^{- p\beta+d(1-p)} \dashint_B \abs{F-\mean{F}_B}^{p'}\,dx, 
  \end{align*}	  
  because then the claim follows by replacing $B$ by $2B$ in all occurring averages.
  To this end, we employ a shift-change (Lemma~\ref{lem:shift-change} applied to $P=0$ and $Q=\nabla u$), which shows that for~$\lambda \in (0,1]$ (to be specified later) there holds
  \begin{align*}
    R 
    &\lesssim \dashint_{B}  \phi^*_{0}\big(\abs{A(\nabla u) - A(\nabla h)}\big) \,dx \\
    &\lesssim \lambda \dashint_B
    \abs{V(\nabla u)}^2\,dx 
    +\lambda^{1-p} \dashint_{B}  \phi^*_{\abs{A(\nabla u)}}\big(\abs{A(\nabla u) - A(\nabla h)}\big) \,dx. 
  \end{align*}
Here the first integral can be bounded by using the degeneracy condition~\eqref{eq:deg-cond-u}, Lemmata~\ref{lem:hammer} and \ref{lem:DKS6.2}, as well as Corollary~\ref{cor:hammer-pge2} (using~$p\geq 2$) which gives
\begin{align*}
	\dashint_B \abs{V(\nabla u)}^2\,dx
	&< \epsilonDG^{-1} \dashint_B \abs{V(\nabla u)- \mean{V(\nabla u)}_B}^2\,dx \\
	&\eqsim \epsilonDG^{-1} \dashint_B \abs{V(\nabla u)- V \left(\mean{\nabla u}_B^A \right)}^2\,dx \\
	&\lesssim \epsilonDG^{-1} \dashint_B \abs{A(\nabla u)- \mean{A(\nabla u)}_B}^{p'}\,dx.
\end{align*}
 In addition, the other integral can be estimated by Lemma~\ref{lem:hammer}, non-linear comparison (Lemma~\ref{lem:nonlin-comparison-V}) and Corollary~\ref{cor:hammer-pge2} again such that we obtain       
    \begin{align*}
	  \dashint_{B}  \phi^*_{\abs{A(\nabla u)}}\big(\abs{A(\nabla u) - A(\nabla h)}\big) \,dx
        &\lesssim \dashint_{B} \abs{V(\nabla u) - V(\nabla h)}^2 \,dx \\
        &\lesssim \dashint_B \phi^*_{\abs{A(\nabla u)}}(\abs{F-\mean{F}_B})\,dx \\
		&\lesssim \dashint_B \abs{F-\mean{F}_B}^{p'}\,dx.
  \end{align*} 
  Hence, we have shown that
  \begin{align*}
  	R &\lesssim \frac{\lambda}{\epsilonDG} \dashint_B \abs{A(\nabla u)- \mean{A(\nabla u)}_B}^{p'}\,dx +
  	\lambda^{1-p} \dashint_B \abs{F-\mean{F}_B}^{p'}\,dx.
  \end{align*}
  Since $p'(1-p)=-p$ choosing $\lambda:=\epsilonDG \, \theta^{p'\beta +d}$ now yields the claimed estimate on~$R$ 
  and thus the proof is complete.
\end{proof}

\subsection{Non-degenerate case}
\label{ssec:non-degenerate-case-u}

Let us now turn to the non-degenerate case. In particular, we will
assume that~$u$ satisfies the following non-degeneracy condition on~$B$
\begin{align}
  \label{eq:non-deg-cond-u}
  \dashint_{B} \abs{V(\nabla u) - \mean{V(\nabla u)}_{B}}^2\,dx &\leq \epsilonDG
  \dashint_{B} \abs{V(\nabla u)}^2\,dx.
\end{align}
Unfortunately, we cannot proceed as in the degenerate case and
compare~$u$ with a $p$-harmonic function~$h$. The reason is a
technical one, namely that the shift-changes cannot be controlled by
means of oscillations.

However, the non-degeneracy condition ensures that $V(\nabla u)$ is in
some sense close to the constant~$\mean{V(\nabla u)}_{B}$. This
implies that~$\nabla u$ is close to $\mean{\nabla u}_{B}^A$. Hence,
the system behaves approximately like a linear one with constant
coefficients. In particular, this argument works best on the set where
$\abs{\nabla u - \mean{\nabla u}_{B}^A} \ll \abs{\mean{\nabla
    u}_{B}^A}$. The non-degeneracy condition however is only in the
integral sense, so there is a small set of points that fail this
condition. It turns out that we can control the critical terms on this
set by the (non-optimal) decay estimates of
Lemma~\ref{lem:decay-u-non-deg-V}. This is done in
Lemma~\ref{lem:decay-A-subset} below. On the remaining ``nice'' set, we will
estimate the $A$-oscillation by using an approximation by a linear
system with constant coefficients, see Lemma~\ref{lem:comparison-linear}.

Before we get to Lemma~\ref{lem:decay-A-subset}, we need a few
auxiliary results on averages.
The subsequent two lemmata follow the spirit of
\cite[Lemma~2.12]{BreCiaDieKuuSch17}.
\begin{lemma}
  \label{lem:B-non-deg}
  There exists a constant~$\epsilon_0$ such that if~$u$ satisfies the
  non-degeneracy condition~\eqref{eq:non-deg-cond-u} on~$B$
  with~$\epsilonDG \leq \epsilon_0$, then
  \begin{align}
    \label{eq:B-non-deg1}
     \frac 12 \max \bigset{ \abs{\mean{\nabla u}_{B}},\,  \abs{\mean{\nabla u}_{B}^A}}
     &\leq \abs{\mean{\nabla u}_{B}^V}
     \leq 2\, \min \bigset{ \abs{\mean{\nabla u}_{B}},
       \abs{\mean{\nabla u}_{B}^A}
       },
  \end{align}
  and
  \begin{align}
    \label{eq:B-non-deg3}
    \max \bigset{ \abs{\mean{\nabla u}_{B}-\mean{\nabla u}_{B}^V}, \, 
    \abs{\mean{\nabla u}^A_{B}-\mean{\nabla u}_{B}^V}}
    &\leq c\, \sqrt{\epsilonDG} \; \abs{\mean{\nabla u}_{B}^V}.
  \end{align}
\end{lemma}
\begin{proof}
  Using~\eqref{eq:non-deg-cond-u} we estimate
  \begin{align*}
    \dashint_B \abs{V(\nabla u) - \mean{V(\nabla u)}_B}^2\,dx 
&\leq \epsilonDG \dashint_B \abs{V(\nabla u)}^2\,dx
    \\
&\leq 2\, \epsilonDG \dashint_B \abs{V(\nabla u) -  \mean{V(\nabla
      u)}_B }^2\,dx +2\, \epsilonDG
      \abs{\mean{V(\nabla u)}_B}^2.
  \end{align*}
  For~$\epsilonDG\leq \epsilon_0 \leq \frac 14$ we can absorb the first term on the
  right-hand side and obtain
  \begin{align}
    \label{eq:B-non-deg2}
    \dashint_B \abs{V(\nabla u) - \mean{V(\nabla u)}_B}^2\,dx
    &\leq 4\, \epsilonDG \abs{\mean{V(\nabla u)}_B}^2
	=4\, \epsilonDG \abs{V\!\left( \mean{\nabla u}_B^V \right)}^2.      
  \end{align}

  Now let $Q\in\left\{\mean{\nabla u}_B, \mean{\nabla u}^A_B\right\}$. 
  Then we can use Lemma~\ref{lem:DKS6.2} and \eqref{eq:B-non-deg2} to derive
  \begin{align}
      \abs{V \!\left( \mean{\nabla u}^V_B \right) - V(Q)}^2
      &= \abs{\mean{V(\nabla u)}_B- V(Q)}^2
      \nonumber\\
      &\leq \dashint_B \abs{V(\nabla u) - V(Q)}^2\,dx
      \nonumber\\
      &\leq c\, \dashint_B \abs{V(\nabla u) - \mean{V(\nabla u)}_B}^2\,dx
      \nonumber\\
      &\leq c\, \epsilonDG \abs{V\!\left( \mean{\nabla u}_B^V \right)}^2 \label{eq:B-non-deg7}
  \end{align}
and hence
  \begin{align*}
	(1-c\, \epsilonDG^{1/2})^2 \, \abs{V \!\left( \mean{\nabla u}^V_B \right)}^2 
    \leq \abs{V(Q)}^2 
    \leq (1+c\, \epsilonDG^{1/2})^2 \, \abs{V \!\left( \mean{\nabla u}^V_B \right)}^2.
  \end{align*}
  Now we choose $\epsilon_0 \geq \epsilonDG$ small enough and use that $\abs{V(Q)}=\abs{Q}^{p/2}$ to conclude
    \begin{align*}
  	\tfrac 12 \, \abs{ \mean{\nabla u}^V_B } 
  	\leq \abs{Q} 
  	\leq 2\, \abs{ \mean{\nabla u}^V_B },
  	\qquad Q\in\left\{\mean{\nabla u}_B, \mean{\nabla u}^A_B\right\},
  \end{align*}
  which shows \eqref{eq:B-non-deg1}.

  It remains to prove~\eqref{eq:B-non-deg3}. 
  To this end, let $P:=\mean{\nabla u}_B^V$. Since $p\geq 2$, it then follows from Lemma~\ref{lem:hammer} and~\eqref{eq:B-non-deg7} that
  \begin{align*}
  	\abs{P}^{p-2} \abs{ P - Q}^2 
  	\leq \left( \abs{P} + \abs{Q} \right)^{p-2} \abs{ P - Q}^2
  	\leq c\, \epsilonDG \, \abs{V(P)}^2
  	= c\, \epsilonDG \, \abs{P}^p, 
  \end{align*}
  i.e.,
  \begin{align*}
  	\abs{ P - Q} &\leq c\, \sqrt{\epsilonDG} \, \abs{P}
  \end{align*}
which completes the proof.
\end{proof}

If $\epsilonDG$ is small enough, our non-degeneracy condition
passes over from $B$ to some sub-balls:
\begin{lemma}
  \label{lem:tauB-average}
  For all~$\tau \in (0,1)$ there exists
  $\epsilon=\epsilon(\tau)>0$ with the following property: If $u$ satisfies the
  non-degeneracy condition~\eqref{eq:non-deg-cond-u} on~$B$ with some~$\epsilonDG \leq \epsilon$, then $u$ also satisfies~\eqref{eq:non-deg-cond-u} on $\tau B$
  with~$\epsilonDG$ replaced by~$16\, \tau^{-d} \,\epsilonDG$ and we have
  \begin{align}
    \label{eq:tauB-average1}
    \begin{aligned}
      \tfrac 12 \, \abs{\mean{V(\nabla u)}_{B}} 
      &\leq \abs{\mean{V(\nabla u)}_{\tau B}} 
      \leq 2 \, \abs{\mean{V(\nabla u)}_{B}},
      \\
      \tfrac 12 \, \abs{\mean{\nabla u}_{B}^V} 
      &\leq \abs{\mean{\nabla u}_{\tau B}^V} 
      \leq 2 \, \abs{\mean{\nabla u}_{B}^V}.
    \end{aligned}
  \end{align}
\end{lemma}
\begin{proof}
  Let us show~\eqref{eq:tauB-average1} first. For this purpose we use~\eqref{eq:non-deg-cond-u} on $B$ to estimate
  \begin{align*}
    \bigabs{\mean{V(\nabla u)}_{\tau B}-\mean{V(\nabla u)}_B}^2\, dx
    &\le
      \dashint_{\tau B}\abs{V(\nabla u)-\mean{V(\nabla u)}_B}^2\,
      dx
    \\
    &\le \tau^{-d}\dashint_{B}\abs{V(\nabla u)-\mean{V(\nabla
      u)}_B}^2\, dx
    \\
    &\le\tau^{-d}\,\epsilonDG\,\abs{\mean{V(\nabla u)}_B}^2. 
  \end{align*}
  Now fix $\delta \in (0, \frac 12)$ small enough such that $(1-\delta)^{\frac 2p} \geq \frac 12$
  and $(1+\delta)^{\frac 2p} \leq 2$. Then, for small~$\epsilon \geq \epsilonDG>0$, we
  have $\tau^{-d}\epsilonDG \leq \delta^2$ and therefore
  \begin{align*}
    \bigabs{\mean{V(\nabla u)}_{\tau B}-\mean{V(\nabla u)}_B}
    &\leq
      \delta \, \abs{\mean{V(\nabla u)}_B},
  \end{align*}
  i.e.,
  \begin{align*}
	(1-\delta) \, \abs{\mean{V(\nabla u)}_{B}}
	&\leq \abs{\mean{V(\nabla u)}_{\tau B}} \leq
	(1+\delta) \, \abs{\mean{V(\nabla u)}_{B}}.
  \end{align*}
  Thus, our choice of $\delta$ and the fact that $\abs{\mean{V(\nabla u)}_B} = \abs{\mean{\nabla u}^V_B}^{\frac p2}$ proves~\eqref{eq:tauB-average1}.
  
  Further, we may estimate
  \begin{align*}
    \dashint_{\tau B} \abs{V(\nabla u)- \mean{V(\nabla u)}_{\tau
    B}}^2\,dx
    &\leq
      \dashint_{\tau B} \abs{V(\nabla u)- \mean{V(\nabla u)}_{B}}^2\,dx
    \\
    &\leq
      \tau^{-d} \dashint_{B} \abs{V(\nabla u)- \mean{V(\nabla u)}_{B}}^2\,dx.
  \end{align*}
  We can additionally assume that~$\epsilon$ is so small that~\eqref{eq:B-non-deg2} from the proof of Lemma~\ref{lem:B-non-deg} holds true. Together with the first part of~\eqref{eq:tauB-average1} this yields
  \begin{align*}
    \dashint_{\tau B} \abs{V(\nabla u)- \mean{V(\nabla u)}_{\tau
    B}}^2\,dx
    &\leq
      4\,\tau^{-d} \,\epsilonDG \, \abs{\mean{V(\nabla u)}_B}^2
      \\
    &\leq
      16\,\tau^{-d} \, \epsilonDG \, \abs{\mean{V(\nabla u)}_{\tau B}}^2
      \\
    &\leq
      16\,\tau^{-d} \, \epsilonDG \, \dashint_{\tau B} \abs{V(\nabla u)}^2\,dx
  \end{align*}
which completes the proof.
\end{proof}

The following lemma is an adaptation of~\cite[Lemma~2.19]{BreCiaDieKuuSch17}.
\begin{lemma}
  \label{lem:decay-A-subset}
  Let~$\sigma,\tau \in (0, \frac 14)$. Then there
  exists~$\epsilon=\epsilon(\sigma,\tau)>0$ such that if~$u$ satisfies the
  non-degeneracy condition~\eqref{eq:non-deg-cond-u} on $B$ with~$\epsilonDG
  \leq \epsilon$, then
  \begin{align*}
    \lefteqn{\dashint_B \abs{A(\nabla u) - \mean{A(\nabla
    u)}_B}^{p'}\,  \chi_{\set{\abs{\nabla u -
    \mean{\nabla u}_B^A} \geq \sigma \abs{\mean{\nabla u}_B^A}}} \,dx}
    \qquad
    \\
    &\leq   c\,\sigma^{-2p} \tau^{2\gamma} \bigg(\dashint_{2B}
      \abs{A(\nabla u)- \mean{A(\nabla u)}_{2B}}^{p'} \,dx +
      \tau^{-d-2\gamma}
      \dashint_{2B} \abs{F-\mean{F}_{2B}}^{p'} \,dx \bigg)
  \end{align*}
  with~$\gamma$ from Lemma~\ref{lem:V-decay}.
\end{lemma}
\begin{proof}
  Let $\sigma,\tau \in (0,\frac 14)$.  Then it is possible to cover~$B$
  by a locally finite set of balls~$\tau B_j$, where the~$B_j$ are
  translates of~$B$ with centers within~$B$. In
  particular,~$B \subset \bigcup_j (\tau B_j) \subset 2 B$ and
  $B_j \subset 2B$.  We define
  \begin{align*}
    I_j &:= \dashint_{\tau B_j} \abs{A(\nabla u) - \mean{A(\nabla
          u)}_B}^{p'}\,  \chi_{E_\sigma} \,dx
  \end{align*}
  with $E_\sigma := \set{x \;|\; \abs{\nabla u - \mean{\nabla u}_B^A} \geq \sigma
    \abs{\mean{\nabla u}_B^A}}$. Then
  \begin{align*}
    I &:= \dashint_B \abs{A(\nabla u) - \mean{A(\nabla
        u)}_B}^{p'}\,  \chi_{\set{\abs{\nabla u -
        \mean{\nabla u}_B^A} \geq \sigma \abs{\mean{\nabla u}_B^A}}}
        \,dx \leq \sum_j \frac{\abs{\tau B_j}}{\abs{B}} I_j. 
  \end{align*}
  If~$\epsilon=\epsilon(\tau)$ is small enough, then according to
  Lemma~\ref{lem:tauB-average} $u$ also satisfies the 
  non-degeneracy condition~\eqref{eq:non-deg-cond-u} w.r.t.\ the ball $\tau B_j$.  Now,
  Lemma~\ref{lem:B-non-deg} for~$B$ and~$\tau B_j$ and
  Lemma~\ref{lem:tauB-average} imply that
  \begin{align}
    \label{eq:8}
    \begin{alignedat}{3}
      \abs{\mean{\nabla u}_{B}^A} &\leq 2\,
      \abs{\mean{\nabla u}_{B}^V} &&\leq 4\,
      \abs{\mean{\nabla u}_{\tau B_j}^V} &&\leq 8\,
      \abs{\mean{\nabla u}_{\tau B_j}^A},
      \\
      \abs{\mean{\nabla u}_{\tau B_j}^A} &\leq 2\,
      \abs{\mean{\nabla u}_{\tau B_j}^V} &&\leq 4\,
      \abs{\mean{\nabla u}_{B}^V} &&\leq 8\,
      \abs{\mean{\nabla u}_{B}^A}.
    \end{alignedat}
  \end{align}
  Therefore, on the set $\tau B_j$ 
  we can estimate
  \begin{align}
    \label{eq:9}
    \begin{aligned}
      \abs{\nabla u - \mean{\nabla u}_B^A} \, \chi_{E_\sigma}
      &\leq \big(\abs{\nabla u - \mean{\nabla u}_{\tau
          B_j}^A} + \abs{\mean{\nabla u}_{B}^A} + \abs{\mean{\nabla
          u}_{\tau B_j}^A} \big) \, \chi_{E_\sigma}
      \\
      &\leq \big( \abs{\nabla u - \mean{\nabla u}_{\tau B_j}^A} + 9\,\abs{\mean{\nabla u}_{B}^A} \big) \, \chi_{E_\sigma}
      \\
    &\leq c\, \sigma^{-1} \abs{\nabla u - \mean{\nabla u}_{\tau B_j}^A} \, \chi_{E_\sigma}
    \end{aligned}
  \end{align}
  since $\sigma < 1$. Moreover, we can employ Lemma~\ref{lem:hammer} (with $P:=\nabla u$ and $Q:=\mean{\nabla u}^A_B$ and $p\geq 2$) to obtain
  \begin{align*}
   \abs{A(\nabla u) - \mean{A(\nabla
    u)}_B}^{p'} \chi_{E_\sigma}
    &\leq c\, \left( \Big(\abs{\mean{\nabla u}^A_B} + \abs{\nabla u - \mean{\nabla
      u}^A_B}\Big)^{p-2}  \abs{\nabla u - \mean{\nabla
      u}^A_B} \right)^{p'} \chi_{E_\sigma}
    \\
    &\leq c\, \bigabs{V(\nabla u) - V(\mean{\nabla
      u}^A_B)}^2\, \bigg( \frac{\abs{\mean{\nabla u}_B^A} +
      \abs{\nabla u -\mean{\nabla
      u}^A_B}}{\abs{\nabla u - \mean{\nabla
      u}^A_B}} \bigg)^{2-p'} \chi_{E_\sigma}
    \\
    &\leq c\,\sigma^{p'-2} \bigabs{V(\nabla u) - V(\mean{\nabla
      u}^A_B)}^2 \chi_{E_\sigma}.
  \end{align*}
  Applying Lemma~\ref{lem:hammer} once more, together with \eqref{eq:8}, \eqref{eq:9} and the fact that $p'-2-p>-2p$ this yields
  \begin{align*}
    \abs{A(\nabla u) - \mean{A(\nabla u)}_B}^{p'} \chi_{E_\sigma}
    &\leq c\, \sigma^{p'-2} \, \phi_{\abs{\mean{\nabla u}^A_B}}(\abs{\nabla u - \mean{\nabla u}^A_B}) \, \chi_{E_\sigma}
    \\
    &\leq c\, \sigma^{p'-2} \, \phi_{\abs{\mean{\nabla
      u}^A_{\tau B_j}}}(c\,\sigma^{-1} \abs{\nabla u - \mean{\nabla
      u}^A_{\tau B_j}})\, \chi_{E_\sigma}
    \\
    &\leq c\, \sigma^{-2p} \, \phi_{\abs{\mean{\nabla
      u}^A_{\tau B_j}}}(\abs{\nabla u - \mean{\nabla
      u}^A_{\tau B_j}}) \, \chi_{E_\sigma}
    \\
    &\leq c\, \sigma^{-2p} \, \abs{V(\nabla u) -
      V(\mean{\nabla u}_{\tau B_j}^A)}^2
  \end{align*}
on every $\tau B_j$ such that
  \begin{align*}
	I &\leq c\,\sigma^{-2p}\, \sum_j \frac{\abs{\tau B_j}}{\abs{B}}
	\dashint_{\tau B_j} \abs{V(\nabla u) - 
		V(\mean{\nabla u}_{\tau B_j}^A)}^2\,dx.  
\end{align*}
  For these local integrals it follows from Lemma~\ref{lem:decay-u-non-deg-V} (applied for $\tfrac 12 B_j$), $\abs{B_j}=\abs{B}$, and
  $B_j \subset 2B$ that
  \begin{align*}
    \lefteqn{\dashint_{\tau B_j}  \abs{V(\nabla u) - \mean{V(\nabla u)}_{\tau
    B}}^{2}\,dx} \qquad
    &
    \\
    &\leq  c\, (2\tau)^{2\gamma} \dashint_{B_j}
      \abs{A(\nabla u)- \mean{A(\nabla u)}_{B_j}}^{p'} \,dx +
      c\,(2\tau)^{-d}
      \dashint_{B_j} \abs{F-\mean{F}_{B_j}}^{p'} \,dx.
    \\
    &\leq  c\, \tau^{2\gamma} \dashint_{2B}
      \abs{A(\nabla u)- \mean{A(\nabla u)}_{2B}}^{p'} \,dx +
      c\,\tau^{-d}
      \dashint_{2B} \abs{F-\mean{F}_{2B}}^{p'} \,dx.
  \end{align*}
  This together with the previous estimate and the covering properties of the $\tau B_j$ proves
  \begin{align*}
    I &\lesssim \sigma^{-2p} \,\sum_j  \frac{\abs{\tau B_j}}{\abs{B}} \bigg( \tau^{2\gamma} \dashint_{2B} \abs{A(\nabla u)- \mean{A(\nabla u)}_{2B}}^{p'} \,dx 
        + \tau^{-d} \dashint_{2B} \abs{F-\mean{F}_{2B}}^{p'} \,dx \bigg) \\
 	&\leq  c\,\sigma^{-2p} \tau^{2\gamma} \bigg(\dashint_{2B}
      \abs{A(\nabla u)- \mean{A(\nabla u)}_{2B}}^{p'} \,dx
      +\tau^{-d-2\gamma} \dashint_{2B} \abs{F-\mean{F}_{2B}}^{p'} \,dx \bigg),
  \end{align*}
as claimed.
\end{proof}

Sometimes it is useful to apply Lemma~\ref{lem:decay-A-subset} with a
different kind of indicator set, namely with the 
$A$-mean value $\mean{\nabla u}^A_B$ replaces by the standard mean
value~$\mean{\nabla u}_B$ on~$B$. The following lemma shows that the two
cases are the same up to a possible change of the constant~$\sigma$.
\begin{lemma}
  \label{lem:change-set}
  For all~$\sigma>0$ there exists~$\epsilon=\epsilon(\sigma)>0$ such
  that if $u$ satisfies the
  non-degeneracy condition~\eqref{eq:non-deg-cond-u} on~$B$
  with~$\epsilonDG \leq \epsilon$, then there holds
  \begin{align*}
    \left\{ x \sep \abs{\nabla u -
    \mean{\nabla u}_B} \geq \sigma \abs{\mean{\nabla u}_B}\right\}
    &\subseteq  \left\{x \sep \abs{\nabla u - \mean{\nabla u}_B^A}
      \geq \frac \sigma  8 \, \abs{\mean{\nabla u}_B^A} \right\}.
  \end{align*}
\end{lemma}
\begin{proof}
  For each
  $x$ with $\abs{\nabla u(x) - \mean{\nabla u}_B} \geq \sigma
    \abs{\mean{\nabla u}_B}$ and small enough~$\epsilon$ we estimate
  with Lemma~\ref{lem:B-non-deg}
  \begin{align*}
    \abs{\nabla u(x) - \mean{\nabla u}_B^A}
    &\geq 
      \abs{\nabla u(x) - \mean{\nabla u}_B} -
      \abs{\mean{\nabla u}_B^A - \mean{\nabla u}_B}
    \\
    &\geq \sigma\,\abs{\mean{\nabla u}_B} - c\, \sqrt{\epsilonDG} \,
    \abs{\mean{\nabla u}_B^V}.
    \\
    &\geq 
      \tfrac \sigma 4\,\abs{\mean{\nabla u}_B^A} - c\, \sqrt{\epsilon}\,
      \abs{\mean{\nabla u}_B^A}.
  \end{align*}
  So for $c \sqrt{\epsilon} \leq \frac \sigma 8$ we get
  \begin{align*}
    \abs{\nabla u(x) - \mean{\nabla u}_B^A} &\geq \frac \sigma 8\,\abs{\mean{\nabla u}_B^A}
  \end{align*}
  which proves the claim.
\end{proof}

In order to proceed towards the desired linear comparison result, let~$z$ denote the solution of the following linear system with
constant coefficients:
\begin{align}
  \label{eq:lin-comparison}
  \begin{alignedat}{2}
    -\divergence\big( (DA)(\mean{\nabla u}_{B}^A)\nabla z\big) &= 0
    &\qquad&\text{in~$B$},
    \\
    z &= u &&\text{on~$\partial B$}.
  \end{alignedat}
\end{align}
We know from linear theory that there exists some constant~$c>0$ such that for any~$\theta\in(0,1)$ we have
\begin{align}
  \label{eq:lin_Decay}
  \sup_{x,x'\in\theta B} \abs{\nabla z(x)-\nabla z(x')}
  \leq c\,\theta \dashint_B \abs{\nabla z-\mean{\nabla z}_B}\, dy.
\end{align} 
For vectors $P,Q  \in \setR^2$ let us define
\begin{align*}
  H(P,Q) &:= A(P) - A(Q) - (DA)(Q) \, (P-Q).
\end{align*}
Then we can use our original
system~\eqref{eq:plap},
\begin{align*}
  -\divergence\big(A(\nabla u)\big) &= - \divergence F,
\end{align*}
to conclude
\begin{align}
  \label{eq:linearized1}
  -\divergence\big( (DA)(\mean{\nabla u}_{B}^A) \nabla (u-z)\big)
  &=
    \divergence\big(H(\nabla
    u, \mean{\nabla u}_{B}^A) \big) - \divergence(F-\mean{F}_B).
\end{align}
In particular, the function~$w := u-z$ satisfies
\begin{alignat*}{2}
  -\divergence\big( (DA)(\mean{\nabla u}_{B}^A) \nabla w\big)
  &=
  \divergence\big(H(\nabla
  u, \mean{\nabla u}_{B}^A) \big) - \divergence(F-\mean{F}_B) &\qquad&
  \text{in $B$},
  \\
  w &= 0 &&\text{on $\partial B$}.
\end{alignat*}
It follows from Lemma~\ref{lem:hammer} for $P:=Q+t\xi$ with~$t \to 0$
and arbitrarily fixed~$Q,\xi \in \setR^2$
that the constant matrix $(DA)(Q)$ satisfies
\begin{align}
  \label{eq:ellipticity}
  c\, \abs{Q}^{p-2} \abs{\xi}^2
    \leq \big((DA)(Q)\, \xi\big) \cdot\xi \leq C\, \abs{Q}^{p-2}  \abs{\xi}^2.
\end{align}
As in \cite{DieLenStrVer12} and~\cite{BreCiaDieKuuSch17} we get the
following comparison estimate.
\begin{lemma}[Linear Comparison]
  \label{lem:comparison-linear}
  Let~$z$ be the solution of~\eqref{eq:lin-comparison}, then
  \begin{align*}
  \abs{\mean{\nabla u}_{B}^A}^{(p-2)p'} \dashint_{B}
        \abs{\nabla u - \nabla z}^{p'}\,dx
      &\leq c\, \dashint_{B} \abs{H(\nabla u, \mean{\nabla
          u}_{B}^A)}^{p'}\,dx + c\, \dashint_B \abs{F-\mean{F}_B}^{p'}\,dx.
   \end{align*}
\end{lemma}
\begin{proof}
  If $w := u-z$, then it is the solution to a linear system of the form
  \begin{align*}
    -\divergence(\mathcal{B} \nabla w) &= - \divergence G
    &&\text{in~$B$},
    \\
    w &= 0 &&\text{on $\partial B$},
  \end{align*}
  where we put $\mathcal{B}:=(DA)(\mean{\nabla u}_{B}^A)$ and $G :=
  -H(\nabla u, \mean{\nabla u}_{B}^A) + (F-\mean{F}_B)$. Due
  to~\eqref{eq:ellipticity} with $Q=\mean{\nabla
    u}_{B}^A$, we have the
  ellipticity condition
  \begin{align*}
    c\, \abs{\mean{\nabla u}_{B}^A}^{p-2} \abs{\xi}^2
    \leq \mathcal{B} \xi \cdot \xi \leq C\, \abs{\mean{\nabla
    u}_{B}^A}^{p-2}  \abs{\xi}^2, \qquad \xi\in\setR^2.
  \end{align*}
  Therefore we can apply the classical $L^{p'}$-regularity result for
  systems with constant coefficients, see \cite[Lemma~2]{DolM95}, to
  conclude
  \begin{align*}
    \bignorm{ \abs{\mean{\nabla u}_{B}^A}^{p-2} \nabla w}_{L^{p'}(B)}
    &\leq c\, \norm{G}_{L^{p'}(B)}.
  \end{align*}
  Thus, the definitions of~$w$ and~$G$ prove the claim.
\end{proof}
In order to estimate the $H$-term in
Lemma~\ref{lem:comparison-linear} we need the following Lemma.
\begin{lemma}
  \label{lem:est-H}
  If $p\geq 2$, then we have 
  \begin{align*}
    \abs{H(P,Q)} &\leq c\, \abs{A(P) - A(Q)} \frac{\abs{P-Q}}{\abs{Q}
                   + \abs{P-Q}}
  \end{align*}
  for all $P,Q \in \setR^2$.
\end{lemma}
\begin{proof}
	We have to distinguish two cases.
	
  \emph{Case $\abs{P-Q} \geq \frac 12 \abs{Q}$:} We can apply Lemma~\ref{lem:hammer} for $p\geq 2$ to conclude
  \begin{align*}
    \abs{H(P,Q)} &\leq \abs{A(P) - A(Q)} + \abs{(DA)(Q)} \, \abs{P-Q}
    \\
                 &\leq \abs{A(P) - A(Q)} + c\, \abs{Q}^{p-2} \abs{P-Q}
    \\
                 &\leq \abs{A(P) - A(Q)} + c\, (\abs{Q} +
                   \abs{P-Q})^{p-2} \abs{P-Q} 
    \\
                 &\leq c\, \abs{A(P) - A(Q)}
                   \\
    &\leq  c\, \abs{A(P) - A(Q)} \frac{\abs{P-Q}}{\abs{Q}
                   + \abs{P-Q}}.
  \end{align*}
  
  \emph{Case $\abs{P-Q} < \frac 12 \abs{Q}$:}  Using $\abs{(D^2 A)(R)}
  \leq c\, \abs{R}^{p-3}$, we obtain with Taylor's formula
  \begin{align*}
    \abs{H(P,Q)} &= \biggabs{\int_0^1 \Big((DA)(Q+t(P-Q)) -
                   (DA)(Q)\Big)\,dt \,(P-Q)}
    \\
    &\leq \int_0^1  t \int_0^1 \abs{(D^2A)(Q -st(P-Q))} \,ds\,dt \,\abs{P-Q}^2 \\
                &\leq c\, \int_0^1 \int_0^1\abs{Q-st(P-Q)}^{p-3}\,ds 
                  \,dt \, \abs{P-Q}^2.
  \end{align*}
  Since $\abs{P-Q} < \frac 12 \abs{Q}$, we have
  $\tfrac 12 \abs{Q} \leq \abs{Q - st(P-Q)} \leq \abs{Q} + \abs{P-Q} \leq 2 \abs{Q}$ for all $s,t\in(0,1)$. Thus, Lemma~\ref{lem:hammer} yields
  \begin{align*}
    \abs{H(P,Q)}       
                        &\leq c\, \big( \abs{Q} + \abs{P-Q} \big)^{p-3} \abs{P-Q}^2
                        \leq  c\, \abs{A(P) - A(Q)} \frac{\abs{P-Q}}{\abs{Q}
                          + \abs{P-Q}}
  \end{align*}
  and the proof is complete.
\end{proof}
We are now prepared to prove the $A$-decay in the non-degenerate case.
\begin{proposition}
  \label{pro:main-A-decay-nondeg}
  Let $\theta \in (0,1)$. Then there
  exist~$\epsilon=\epsilon(\theta)>0$ and $c_\theta>0$ such that if~$u$ satisfies the
  non-degeneracy condition~\eqref{eq:non-deg-cond-u} on $B$ with $\epsilonDG
  \leq \epsilon$, then
  \begin{align*}
    \lefteqn{ \bigg(\dashint_{\theta B }\abs{A(\nabla u)-\mean{A(\nabla u)}_{\theta
    B}}^{p'}\, dx\bigg)^{\frac 1 {p'}} } \qquad
    &
    \\
    &\le 
      c\,\theta\, \bigg(\dashint_{2B}\abs{A(\nabla u)-\mean{A(\nabla
      u)}_{2B}}^{p'}\, dx \bigg)^{\frac 1 {p'}} + c_\theta
      \bigg(\dashint_{2B} \abs{F-\mean{F}_{2B}}^{p'}\, dx\bigg)^{\frac 1 {p'}}.  
  \end{align*}
\end{proposition}
\begin{proof}
  According to Lemma~\ref{lem:osc-aux} we have
  \begin{align*}
  	\mathrm{I} &:= \dashint_{\theta B} \abs{A(\nabla u) - \mean{A(\nabla u)}_{\theta
  			B}}^{p'}\,dx
  	\\
  	&\lesssim
  	\dashint_{\theta B} \abs{A(\nabla u) - \mean{A(\nabla u)}_{\theta
  			B}}^{p'} \chi_{\set{\abs{\nabla u -
  				\mean{\nabla u}_{\theta B}} \geq \sigma \abs{\mean{\nabla u}_{\theta B}}}} \,dx
  	\\
  	&\qquad 
  	+     \dashint_{\theta B} \abs{A(\nabla u) - A(\mean{\nabla u}_{\theta
  			B})}^{p'} \chi_{\set{\abs{\nabla u -
  				\mean{\nabla u}_{\theta B}} < \sigma \abs{\mean{\nabla u}_{\theta B}}}} \,dx
  	\\
  	&=: \mathrm{II} + \mathrm{III}.
  \end{align*}
  If~$\epsilonDG$ is sufficiently small (depending on~$\theta$), it
follows from Lemma~\ref{lem:tauB-average} that~$u$ also satisfies
the non-degeneracy condition~\eqref{eq:non-deg-cond-u}
on~$\theta B$. Therefore, we get by Lemmata~\ref{lem:change-set}, \ref{lem:decay-A-subset}, and \ref{lem:osc-aux} for any~$0<\tau < \tfrac 14$
  \begin{align*}
    \mathrm{II} 
    &\leq \dashint_{\theta B} \abs{A(\nabla u) - \mean{A(\nabla u)}_{\theta B}}^{p'} \chi_{ \left\{\abs{\nabla u - \mean{\nabla u}^A_{\theta B}} \geq \tfrac \sigma 8
         \abs{\mean{\nabla u}^A_{\theta B}}\right\}} \,dx
    \\
             &\lesssim \sigma^{-2p} \tau^{2\gamma} \bigg(\dashint_{2
               \theta B}
      \abs{A(\nabla u)- \mean{A(\nabla u)}_{2\theta B}}^{p'} \,dx +
      \tau^{-d-2\gamma}
      \dashint_{2 \theta B} \abs{F-\mean{F}_{2 \theta B}}^{p'} \,dx \bigg)
    \\
             &\lesssim \sigma^{-2p} \tau^{2\gamma} \theta^{-d} \bigg(\dashint_{2
               B}
      \abs{A(\nabla u)- \mean{A(\nabla u)}_{2 B}}^{p'} \,dx +
      \tau^{-d-2\gamma}
      \dashint_{2 B} \abs{F-\mean{F}_{2B}}^{p'} \,dx \bigg).
  \end{align*}
  Later we will take~$\tau=\tau(\sigma,\theta)$ small such that the factor in front of the $A$-oscillation on~$2B$ is small. Of course, the price will be a large
  factor in front of the $F$-oscillation.

  Let us now estimate~ $\mathrm{III}$. We apply Lemma~\ref{lem:hammer} and use that  $\sigma \leq 1$ and $p\geq 2$ to obtain
  \begin{align*}
    \mathrm{III} &\lesssim \!\dashint_{\theta B} \!\!\Big[ (\abs{ \mean{\nabla u}_{\theta B}}
         + \abs{\nabla u \!-\! \mean{\nabla u}_{\theta B}})^{p-2} \abs{
                     \nabla u \!-\! \mean{\nabla u}_{\theta B}} \Big]^{p'}  \!
                     \chi_{\set{\abs{\nabla u -
                     \mean{\nabla u}_{\theta B}} < \sigma \abs{\mean{\nabla u}_{\theta B}}}} dx
    \\
                   &\lesssim \abs{ \mean{\nabla u}_{\theta
                     B}}^{(p-2)p'}
                     \dashint_{\theta B}  \abs{
                     \nabla u - \mean{\nabla u}_{\theta B}}^{p'} 
                     \,dx.
  \end{align*}
  The latter integral can be estimated further in terms of the solution~$z$ of the linearized equation~\eqref{eq:lin-comparison}. Indeed, from the linear decay~\eqref{eq:lin_Decay} of solutions to linear systems and $\theta^{p'}<1<\theta^{-d}$ it follows
  \begin{align*}
  	\dashint_{\theta B}  \abs{\nabla u - \mean{\nabla u}_{\theta B}}^{p'} \,dx
  	&\lesssim \dashint_{\theta B}  \abs{\nabla z - \mean{\nabla
     z}_{\theta B}}^{p'} \,dx + \dashint_{\theta B}  \abs{\nabla u -
     \nabla z}^{p'} \,dx
    \\
  	&\lesssim \sup_{x,x' \in \theta B}  \abs{\nabla z(x) - \nabla
     z(x')}^{p'} + \dashint_{\theta B}  \abs{\nabla u -
     \nabla z}^{p'} \,dx
    \\
  	&\lesssim \theta^{p'} \dashint_{B}  \abs{\nabla z -
     \mean{\nabla z}_{B}}^{p'} \,dx + \theta^{-d}\dashint_{B}
     \abs{\nabla u - \nabla z}^{p'} \,dx
    \\
  	&\lesssim \theta^{p'} \dashint_{B}  \abs{\nabla u - \mean{\nabla u}_{B}}^{p'} \,dx + \theta^{-d}\dashint_{B}  \abs{\nabla u - \nabla z}^{p'} \,dx.
  \end{align*}
 In addition, Lemma~\ref{lem:tauB-average} and
  Lemma~\ref{lem:B-non-deg} imply
  \begin{align*}
    \abs{\mean{\nabla u}_{\theta
          B}} \eqsim \abs{\mean{ \nabla u}_B} \eqsim  \abs{\mean{ \nabla u}^A_B}.
  \end{align*}
  Since $p\geq 2$ we can use this estimate to conclude that
  \begin{align*}
    \mathrm{III} 
        &\lesssim \theta^{p'} \, \abs{\mean{\nabla u}_{B}^A}^{(p-2)p'}  \dashint_{B}  \abs{\nabla u -
          \mean{\nabla u}_{B}}^{p'} \,dx 
+ \theta^{-d} \, \abs{\mean{\nabla u}_{B}^A}^{(p-2)p'}   \dashint_{B}  \abs{\nabla u -
          \nabla z}^{p'} \,dx
    \\
    &=: \mathrm{III}_1+\mathrm{III}_2. 
  \end{align*}
  Now we apply Lemmata~\ref{lem:osc-aux} and \ref{lem:hammer} with $p \geq 2$ to obtain the following bound on~$\mathrm{III}_1$: 
  \begin{align}
      \mathrm{III}_1  
      &\lesssim \theta^{p'} \, \abs{\mean{\nabla
          u}^A_{B}}^{(p-2)p'}  \dashint_{B}  \abs{\nabla u - 
        \mean{\nabla u}_{B}^A}^{p'} \,dx   
      \nonumber\\
      &\lesssim \theta^{p'} \,   \dashint_{B}  \abs{A(\nabla u) - A(
        \mean{\nabla u}_{B}^A)}^{p'} \,dx
      \nonumber\\
      &= \theta^{p'} \,   \dashint_{B}  \abs{A(\nabla u) - \mean{A(
        \nabla u)}_{B}}^{p'} \,dx.
          \label{eq:III_1}
  \end{align}
  
  Next, for $\mathrm{III}_2$ we use the linear comparison of
  Lemma~\ref{lem:comparison-linear} and the estimate for~$H$ from
  Lemma~\ref{lem:est-H} to deduce
  \begin{align*}
    \mathrm{III}_2 &\le \theta^{-d} c \dashint_{B} \abs{H(\nabla
            u, \mean{\nabla u}_{B}^A)}^{p'}\,dx + \theta^{-d}
            \dashint_B \abs{F-\mean{F}_B}^{p'}\,dx 
    \\
          &\lesssim \theta^{-d} \dashint_B \abs{A(\nabla u)-A(\mean{\nabla
            u}^A_{B})}^{p'}
            \bigg( \frac{\abs{\nabla u-\mean{\nabla
            u}^A_{B}}}{\abs{\mean{\nabla u}_{B}^A}+\abs{\nabla
            u-\mean{\nabla 
            u}^A_{B}}} \bigg)^{p'}\, dx
+ \theta^{-d}
            \dashint_{B}\abs{F-\mean{F}_B}^{p'}\, dx
    \\
          &=: \mathrm{III}_{2}^A + \mathrm{III}_{2}^F,
  \end{align*}
where for $\sigma,\tau <1$
  \begin{align*}
	\mathrm{III}_2^F 
	&\lesssim \sigma^{-2p} \tau^{-d} \theta^{-d}  \dashint_{2B}\abs{F-\mean{F}_{2B}}^{p'}\, dx.
\end{align*}
  Note that on the set
  $\set{x \sep \abs{\nabla u - \mean{\nabla u}_{B}^A} < \sigma
    \abs{\mean{\nabla u}_{B}^A}}$ the fraction in the integral
  of~$\mathrm{III}_{2}^A$ is smaller than~$\sigma$, while on its complement it is 
  bounded by 
  one. 
  Hence,
  \begin{align*}
    \mathrm{III}_{2}^A 
    &\leq \theta^{-d} \dashint_B \abs{A(\nabla u)-\mean{A(\nabla
      u)}_{B}}^{p'} \chi_{\set{ \abs{\nabla u - \mean{\nabla
      u}_{B}^A} \geq \sigma \abs{\mean{\nabla u}_{B}^A}}}\,dx
    \\
    &\quad + \sigma^{p'}\,\theta^{-d} \dashint_B\abs{A(\nabla u)- \mean{A(\nabla
      u)}_{B} }^{p'}\, \chi_{\set{ \abs{\nabla u - \mean{\nabla
      u}_{B}^A} < \sigma \abs{\mean{\nabla u}_{B}^A }}}\,dx,
  \end{align*}
  where we have also used $\mean{A(\nabla u)}_{B} = A(\mean{\nabla u}^A_{B})$.
  Clearly, the second integral is small for small~$\sigma>0$. 
  Moreover, the first one can be estimated further by Lemma~\ref{lem:decay-A-subset} provided that~$\epsilon = \epsilon(\sigma,\tau) \geq \epsilonDG$ is small enough.
  Overall, in combination with~\eqref{eq:III_1} we arrive at
  \begin{align*}
  	\mathrm{III}
  	&\lesssim \mathrm{III_1} + \mathrm{III}_{2}^A + \mathrm{III}_{2}^F 
  	\\
    &\lesssim  \theta^{p'} \, \dashint_{B}  \abs{A(\nabla u) - \mean{A(
    		\nabla u)}_{B}}^{p'} \,dx \\
     &\quad + \sigma^{-2p} \tau^{2\gamma} \theta^{-d} \bigg(\dashint_{2B}
      \abs{A(\nabla u)- \mean{A(\nabla u)}_{2B}}^{p'} \,dx 
      + \tau^{-d-2\gamma} \dashint_{2B} \abs{F-\mean{F}_{2B}}^{p'} \,dx \bigg)
    \\
    &\quad + \sigma^{p'} \theta^{-d} \dashint_B\abs{A(\nabla u)-\mean{A(\nabla
      u)}_{B}}^{p'}\,dx. 
\end{align*}
  Recall that also
  \begin{align*}
    \dashint_{B}  \abs{A(\nabla u) - \mean{A(
    \nabla u)}_{B}}^{p'} \,dx &\leq c\,
                                \dashint_{2B}  \abs{A(\nabla u) - \mean{A(
                                \nabla u)}_{2B}}^{p'} \,dx.
  \end{align*}
  Thus, combing the estimates for $\mathrm{II}$ and $\mathrm{III}$ we obtain our final estimate
  \begin{align*}
    \mathrm{I}
    &\lesssim \Big(\theta^{p'} + \sigma^{-2p} \tau^{2\gamma} \theta^{-d} + \sigma^{p'}\theta^{-d} \Big)\,   \dashint_{2B}  \abs{A(\nabla u) - \mean{A(
      \nabla u)}_{2B}}^{p'} \,dx
    \\
    &\quad + \sigma^{-2p} \tau^{-d} \theta^{-d}
         \dashint_{2B} \abs{F-\mean{F}_{2B}}^{p'} \,dx.
  \end{align*}
  Now, we choose the parameters in the following order: Given $\theta\in(0,1)$, we first
  choose~$\sigma= \sigma(\theta)>0$ small enough such that
  $\sigma^{p'}\theta^{-d} \leq \theta^{p'}$. Second, we choose
  $\tau=\tau(\theta,\sigma,\gamma) = \tau(\theta)>0$ small such that also 
  $\sigma^{-2p} \tau^{2\gamma} \theta^{-d} \leq
  \theta^{p'}$. Moreover, for the validity of the above estimates
  (applicability of Lemma~\ref{lem:decay-A-subset}) we have to choose
  $\epsilon= \epsilon(\sigma,\tau)=\epsilon(\theta)>0$ small. Taking
  the $p'$-root then proves our desired $A$-decay.
\end{proof}

\subsection{Proof of Theorems~\ref{thm:osc-estimates} and~\ref{thm:osc-estimates-general}}
\label{ssec:proof-theorem-osc-est}
We will now combine the estimates for the degenerate and the
non-degenerate case to prove the main results of this section, namely
Theorems~\ref{thm:osc-estimates} and~\ref{thm:osc-estimates-general}.
\begin{proof}[Proof of Theorem~\ref{thm:osc-estimates}]
  Let $\beta \in (0,1)$. Define~$\beta_2 := \frac{1+\beta}{2}$ such that
  $\beta < \beta_2 < 1$.  We will combine
  Propositions~\ref{pro:decay-u-deg} (with~$\beta_2$) and~\ref{pro:main-A-decay-nondeg} to prove our claim. 
  To this end, we first choose~$\theta \in (0,1)$ so small such that $\theta_0:=\tfrac{1}{2}\theta$ satisfies
  $c\, \theta^{\beta_2} + c\, \theta \leq \theta_0^\beta$, where
  $c\,\theta^{\beta_2}$ is from
  Proposition~\ref{pro:decay-u-deg} and $c\, \theta$ from
  Proposition~\ref{pro:main-A-decay-nondeg}. 
  In particular this determines~$\epsilon=\epsilon(\theta)$ in
  Proposition~\ref{pro:main-A-decay-nondeg}. 
  
  If $u$ satisfies the
  non-degeneracy condition~\eqref{eq:non-deg-cond-u} on~$\frac 12 B$,
  then the claim follows by
  Proposition~\ref{pro:main-A-decay-nondeg}. If, however,~\eqref{eq:non-deg-cond-u} is not
  satisfied on~$\frac 12 B$, then we can apply
  Proposition~\ref{pro:decay-u-deg} to deduce our claim.
\end{proof}

\begin{proof}[Proof of Theorem~\ref{thm:osc-estimates-general}]
  It follows by repeated use of Theorem~\ref{thm:osc-estimates} that for our fixed~$\theta_0$ and all $k\in\setN$ there holds
  \begin{align}
    \label{eq:osc-est-gen-1}
    \osc_{p'} A(\nabla u)(x,\theta_0^k t) \leq \theta_0^{k \beta}
    \osc_{p'} A(\nabla u)(x,t) + c \sum_{j=0}^{k-1}
    \theta_0^{(k-1-j)\beta} \osc_{p'} F(x, \theta_0^j t).
  \end{align}
  Now, let $\theta \in (0,1)$. Then we
  find~$k \in \setN_0$ such that
  $\theta_0^{k+1} \leq \theta \leq \theta_0^k$.  
  Moreover, note that for
  general~$G \in L^{p'}$, any $\lambda \in [\theta_0,1]$, and $s>0$ we have 
  \begin{align}
    \label{eq:osc-est-gen-2}
    \osc_{p'}G(x, \theta_0 s) \lesssim
    \osc_{p'}G(x, \lambda s) \lesssim 
    \osc_{p'}G(x, s)
  \end{align}
  with constants only depending on
  the fixed~$\theta_0$.  Thus, the claim follows
  from~\eqref{eq:osc-est-gen-1} in a standard way by changing the
  discrete sum by an integral using~\eqref{eq:osc-est-gen-2}. This
  step also introduces the constant in front of~$\theta^\beta$.
\end{proof}

\subsection{Consequences and remarks}
\label{ssec:consequences}

In this section we present a few consequences of
Theorem~\ref{thm:osc-estimates}.  
Let us begin with how our estimates
improve the results in~\cite{BreCiaDieKuuSch17}, where pointwise
regularity estimates have been proven for the system version
of~\eqref{eq:plap} with $1<p < \infty$ and~$\Omega \subset \Rd$. As an
important intermediate step they prove an assertion very similar to our 
Theorem~\ref{thm:osc-estimates}. For the case~$p \geq 2$ this result
reads as follows:
\begin{proposition}[{\cite[Proposition~2.1]{BreCiaDieKuuSch17}}]
  \label{pro:5authors}
  For~$\Omega \subset \setR^d$ and $p\geq 2$ let $\gamma$ denote the best possible exponent in Lemma~\ref{lem:V-decay}. If $u$ solves~\eqref{eq:plap}, then for $\beta \in (0,\gamma)$ there exists~$\theta_0 \in (0,1)$ and $c_\beta>0$ such that\footnote{It
  	is stated in~\cite[Proposition~2.1]{BreCiaDieKuuSch17} that $\beta
  	\in (0, \min \set{1, \gamma \frac{2}{p'}})$, but this is a typo. It
  	should be $\beta
  	\in (0, \min \set{1, \gamma \frac{2}{\overline{p}'}})$ with $\overline{p} =
  	\min \set{p,2}$.}
  \begin{align*}
   & \bigg(\dashint_{\theta_0 B} \abs{A(\nabla u) - \mean{A(\nabla u)}_{\theta_0
    B}}^{p'} \,dx \bigg)^{\frac 1{p'}}
\\
    &\qquad\leq \theta_0^\beta \bigg(
    \dashint_{B} \abs{A(\nabla u) - \mean{A(\nabla u)}_{
    		B}}^{p'} \,dx \bigg)^{\frac 1{p'}}
    + c_\beta\, \bigg(
      \dashint_{B} \abs{F - \mean{F}_{
      B}}^{p'} \,dx \bigg)^{\frac 1{p'}}.
  \end{align*}
\end{proposition}
In contrast to this, our Theorem~\ref{thm:osc-estimates} improves the condition on~$\beta$ from $\beta \in (0,\gamma)$ to
$\beta \in (0,1)$. However, our result is restricted to the scalar problem in the plane.

Note that the best value of~$\gamma>0$ is not known for
systems or higher dimensions. 
Even for the scalar case in the plane it remains open whether the assertion extends to the limiting case $\beta=1$.
See, in particular, the discussion in Subsection~\ref{ssec:open-problem}. Therefore, it
was not possible for us to directly apply
Proposition~\ref{pro:5authors} but rather make use of the improved
$A$-decay of Theorem~\ref{thm:h-decay}.

Let us remark that all the results of Subsection~\ref{sec:oscill-estim}
are only based on the $A$-decay of Theorem~\ref{thm:h-decay}. In
particular, if the $p$-harmonic system~\eqref{eq:pharmonic} is proven
to satisfy an $A$-decay with power~$\beta$, then
Theorem~\ref{thm:osc-estimates} remains valid for all (!) smaller
exponents. Hence, all the consequences below will remain valid. 
Especially all proofs of Subsection~\ref{sec:oscill-estim} are valid
for any dimension as well as for systems.

At this point it is also worth to mention that for~$p\geq 2$ a $V$-decay with
exponent~$\gamma$ directly implies an $A$-decay for any
exponent $\beta \in (0,\gamma)$. This was proven
in~\cite[Remark~5.6]{DieKapSch11}.  Hence, our method is more flexible, since we only need~$A$-decay.

Let us now present some results which improve Theorem~1.3 and Corollary~5.2 of~\cite{BreCiaDieKuuSch17}. Particularly, we extend the range of admissible~$\beta$'s to $\beta \in (0,1)$.

\begin{corollary}[General oscillation estimate]
  \label{cor:general-osc-estimate}
  Let $\Omega \subset \setR^2$ and $p \geq 2$. Assume that 
  $u\colon \Omega \to \setR$ satisfies~\eqref{eq:plap}.
  If for $R>0$, some $\beta\in(0,1)$, and~$\omega\colon (0,R)\to (0,\infty)$ the function~$r \mapsto \omega(r)\,r^{-\beta}$ is almost decreasing\footnote{That is, we assume that $\omega(r)\, r^{-\beta} \le c\, \omega(\rho) \,\rho^{-\beta}$ for all $0 < r < \rho< R$.} in~$(0,R)$, then for any ball~$B_R$ with $2B_R \subset \Omega$ there holds
  \begin{align*}
    M_{\omega,R}^{\sharp,1}(A(\nabla u))(x)
    &\leq c_\beta\,     M_{\omega,R}^{\sharp,p'}(F)(x)
      + c_\beta\,   \frac{1}{\omega(R)} \bigg(
      \dashint_{2 B_R} \abs{A(\nabla u) - \mean{A(\nabla u)}_{
      2 B_R}}^{p'} \,dx \bigg)^{\frac 1{p'}}.
  \end{align*}
\end{corollary}
Here for any $q\geq 1$ the localized fractional sharp maximal operator~$ M_{\omega,R}^{\sharp,q}$ of $f\in L_{\loc}^q(\setR^2)$ is defined pointwise by
  \begin{align*}
    M_{\omega,R}^{\sharp,q}(f)(x) &:= \sup_{
                                    \substack{B_r \ni x
    \\
    r <
    R}} \frac 1{\omega(r)} \bigg(\dashint_{B_r}
    \abs{f-\mean{f}_{B_r}}^q\, dy\bigg)^{\frac 1q}. 
  \end{align*}

\begin{corollary}
  \label{cor:general-osc-estimate-Cs}
  For $\Omega \subset \setR^2$ and $p \geq 2$ let 
  $u\colon \Omega \to \setR$ satisfy~\eqref{eq:plap}. If~$s \in (0,1)$ and
  $2B \subset \Omega$, then
  \begin{align*}
    \abs{A(\nabla u)}_{C^s(B)}\le
    c_s \Big( \abs{F}_{C^s(2 B)}+\dashint_{2B} \abs{A(\nabla u)
    - \mean{A(\nabla u)}_{2B}}^{p'}\,dx \bigg)^{\frac{1}{p'}}\Big).
  \end{align*}
\end{corollary}
The proofs of these results are only based on Proposition~\ref{pro:5authors},
resp. \cite[Proposition~2.1]{BreCiaDieKuuSch17}. 
Our improved version
in Theorem~\ref{thm:osc-estimates} immediately implies the claims.

Note that Corollary~\ref{cor:general-osc-estimate-Cs} is also
covered by Theorem~\ref{thm:reg-transfer} below since $C^s = \bfB^s_{\infty,\infty}$.

\section{Regularity transfer - nonlinear \Calderon{}-Zygmund estimates}
\label{sec:regularity-transfer}
In this section we show that Sobolev regularity up to order one
transfers from the right-hand side $F$ to the flux $A(\nabla u)$.  
We present this result in more general scales of Besov and Triebel-Lizorkin spaces in Subsection~\ref{ssec:regul-transf-from}. 
At first, in Theorem~\ref{thm:reg-transfer} it is shown that in terms of quasi-semi norms we have
\begin{align*}
\abs{A(\nabla u)}_{\bfX(B)} \le C\,\abs{F}_{\bfX(2B)} + \texttt{lower order terms of~$A(\nabla u)$},
\end{align*}
where ~$\bfX$ stands for either~$\bfB^s_{\rho,q}$ or~$\bfF^s_{\rho,q}$. Afterwards, in Subsection~\ref{ssec:transfer-nabla-u}, we
study how this new regularity
for~$A(\nabla u)$ translates into regularity statements for~$\nabla u$ and~$V(\nabla u)$.

\subsection{Regularity transfer from~\texorpdfstring{$F$}{F} to~\texorpdfstring{$A(\nabla u)$}{A(nabla u)}}
\label{ssec:regul-transf-from}

For a ball~$B$ let us denote by $\bfB^s_{\rho,q}(B)$ and $\bfF^s_{\rho,q}(B)$, the Besov space, resp.\ Triebel-Lizorkin space, of functions on~$B$ with
differentiability~$s>0$, integrability~$0<\rho \leq \infty$, and fine index~$0<q\leq \infty$ (with $\rho<\infty$ for the $\bfF$-scale).  We use
$\norm{\cdot}_{\bfB^s_{\rho,q}(B)}$ to denote the (quasi-) norm and
$\abs{\cdot}_{\bfB^s_{\rho,q}(B)}$ for the (quasi-) semi norm describing the part of the $s$-order derivatives. Likewise we do for the $\bfF$-scale. 
The exact definitions are introduced below. As usual, we let $(x)_+:=\max\{0,x\}$ for $x\in\setR$. 
Then our main result is the following local regularity transfer:
\begin{theorem}
  \label{thm:reg-transfer}
  Given $2\leq p < \infty$, some domain~$\Omega \subset \setR^2$, and $F \in L^{p'}(\Omega)$
  let $u \in W^{1,p}(\Omega)$ be a (scalar) weak solution to
  \begin{align*}
  -\divergence(A(\nabla u)) &= -\divergence F \qquad \text{in $\Omega$}.
  \end{align*}
  Further, let $s >0$ and
  $\rho,q \in (0,\infty]$ be such that
  \begin{align}
    \label{eq:ref-transfer-cond}
    d \bigg( \frac{1}{\rho} -\frac{1}{p'} \bigg)_+ < s < 1,
  \end{align}
  i.e., $\bfB^s_{\rho,q}(B) \compactembedding L^{p'}(B)$. Then for any
  ball~$B$ with $2B \subset \Omega$ there holds
  \begin{align}
    \label{eq:ref-transfer-est}
    \abs{A(\nabla u)}_{\bfB^s_{\rho,q}(B)} \lesssim
    \abs{F}_{\bfB^s_{\rho,q}(2B)} + \bigg(
    \dashint_{2B} \abs{A(\nabla u)
    - \mean{A(\nabla u)}_{2B}}^{p'}\,dx \bigg)^{\frac{1}{p'}}. 
  \end{align}
  If additionally~$\rho < \infty$ and
  \begin{align}
    \label{eq:ref-transfer-cond2}
    d \bigg( \frac{1}{q} -\frac{1}{p'} \bigg)_+ < s < 1,
  \end{align}
  then the same estimate~\eqref{eq:ref-transfer-est} holds
  true when $\bfB^s_{\rho,q}$ is replaced by~$\bfF^s_{\rho,q}$.
\end{theorem}
\begin{remark}
  \label{rem:reg-transfer-Rd}
  Our result also generalizes to higher dimensions and
  vectorial solutions. However, in this setting the
  differentiability is restricted to~$s \in (0,\beta_0)$,
  where~$\beta_0\leq 1$ is some unknown small number. The reason behind this
  is the worse decay estimate for $p$-harmonic functions in this
  general situation.  In fact, our regularity transfer holds exactly for
  the same range as the decay estimates. See the also the 
  Remarks~\ref{rem:h-decay-Rd} and~\ref{rem:osc-est-Rd} for more
  details.
\end{remark}
Let us introduce the norms used to describe our spaces
$\bfB^s_{\rho,q}(B)$ and $\bfF^s_{\rho,q}(B)$, respectively. In order to have the
constants in~\eqref{eq:ref-transfer-est} independent of the chosen
ball~$B=B_R$, we are using norms that are invariant with respect to the
scaling $g\mapsto g_R := g(\cdot/R)$. For a ball~$B$ we introduce the scaling invariant~$L^\rho(B)$ (quasi-) norm by
\begin{align*}
  \normm{f}_{L^\rho(B)} := 
  \begin{cases}
  	\bigg( \displaystyle\dashint_B \abs{f}^\rho\,dx
  	\bigg)^{\frac 1\rho} &\text{ for} \quad	0< \rho < \infty,\\
  	\norm{g}_{L^\infty(B)} & \text { if} \quad \rho=\infty.
  \end{cases}
\end{align*}
Moreover, we need a localized version of the oscillations from
Section~\ref{sec:oscill-estim}.  For every $g \in L^w(B)$ with
$1 \leq \omega < \infty$ we define its localized oscillation by
\begin{align*}
  \osc^B_wg(x,t) &:= \bigg( \dashint_{B_t(x) \cap B} \abs{g -
                  \mean{g}_{B_t(x) \cap B}}^w \,dy
                \bigg)^{\frac 1w}.
\end{align*}
Using these ingredients there holds the following characterization of~$\bfB^s_{\rho,q}(B)$ and
$\bfF^s_{\rho,q}(B)$, resp., which for simplicity we could take here also  as
a definition.

\begin{lemma}[Characterization by oscillations]
  \label{lem:characterization}
  Let $B \subset \Rd$ be a ball with radius~$R$.  Further let
  $0<\rho,q\leq \infty$, and $s>0$, as well as
  $1\leq w \leq \infty$, and assume that
  \begin{align*}
    d \left( \frac{1}{\rho} - \frac{1}{w}\right)_+ < s < 1,
  \end{align*}
  i.e., $\bfB^s_{\rho,q}(B) \compactembedding L^w(B)$, resp.
  $\bfF^s_{\rho,q}(B) \compactembedding L^w(B)$.
  \begin{enumerate}
  \item
    \label{itm:characterization-B}
    Then there holds
    \begin{align*}
      \bfB_{\rho,q}^s(B) = \bigset{ g \in L^{\max\{\rho,w\}}(B)
      \,|\, \norm{g}_{\bfB_{\rho,q}^s(B)} <\infty },
    \end{align*}
    with $\norm{g}_{\bfB_{\rho,q}^s(B)}
      := \normm{g}_{L^\rho(B)} + \abs{g}_{\bfB_{\rho,q}^s(B)}$,
      where 
    \begin{align*}
      \abs{g}_{\bfB_{\rho,q}^s(B)} 
      &:= R^s\Bigg( \int_0^R \bigg(
                                \frac{\normm{\osc_w^B g(\cdot,t)
                                }_{L^\rho(B)}}{t^s}
                                \bigg)^q \frac{dt}{t} \Bigg)^{\frac 1q}
    \end{align*}
    with the usual modification for $q=\infty$.
  \item
    \label{itm:characterization-F}
    Additionally, assume~$\rho < \infty$ and
    \begin{align*}
      d \left( \frac{1}{q} - \frac{1}{w}\right)_+<s,
    \end{align*}
    Then there holds
    \begin{align*}
      \bfF_{\rho,q}^s(B) = \bigset{ g \in L^{\max\{\rho,w\}}(B)
      \,|\, \norm{g}_{\bfF_{\rho,q}^s(B)} <\infty },
    \end{align*}
    with
    $\norm{g}_{\bfF_{\rho,q}^s(B)} := \normm{g}_{L^\rho(B)} +
    \abs{g}_{\bfF_{\rho,q}^s(B)} $, where
    \begin{align*}
      \abs{g}_{\bfF_{\rho,q}^s(B)} 
      &:= R^s \Biggnormm{ \Bigg( \int_0^R \bigg(
        \frac{\osc_w^B g(\cdot,t)
        }{t^s}
        \bigg)^q \frac{dt}{t} \Bigg)^{\frac
        1q}         }_{L^\rho(B)}
    \end{align*}
    (modification for $q=\infty$). 
  \end{enumerate}
\end{lemma}
\begin{proof}
  For scalar functions this characterization is a special case of
  Triebel~\cite[Theorem~2.2.2 and (2.22)]{Tri1989} for
  bounded~$C^\infty$ domains.  Our scaling invariant version can be
  obtained by applying the result of Triebel to the unit ball~$B_1(0)$
  and then scale by $g_R(x) := g(x/R)$ and translate. The extension
  for vector-valued functions is straightforward.
\end{proof}

Note that the stated (quasi-) semi norm implicitly
depends on the parameter~$w$. This however only gives rise to equivalent (quasi-) norms in the same space.
Moreover, let us recall that the Besov and Triebel-Lizorkin scales of smoothness $s>0$ include (among others) the more familiar H\"older-Zygmund spaces $C^s=\bfB^s_{\infty,\infty}$, Sobolev-Slobodeckij spaces $W^{s,\rho} = \bfB^s_{\rho,\rho}=\bfF^s_{\rho,\rho}$ (with $1\leq \rho<\infty$ and $s\notin\setN$), and Bessel-potential spaces $H^{s,\rho}=\bfF^s_{\rho,2}$ (with $1<\rho<\infty$). For details we refer to \cite{RunSic96}.

Let us now prove our main result of this
Section~\ref{sec:regularity-transfer}.

\begin{proof}[Proof of Theorem~\ref{thm:reg-transfer}]
  For simplicity of presentation we first prove the result in the
  situation of Banach spaces. That is, for now we assume~$\rho,q \geq 1$. 
  The modifications needed for the quasi-Banach case, where the
  used quasi-triangle inequalities produce additional constants, are explained afterwards.
  
  Let us choose~$\beta \in (s,1)$. Then according to
  Theorem~\ref{thm:osc-estimates} we find~$\theta_0<1$ and some
  constant~$c_\beta$ such that the decay estimate
  \begin{align}
    \label{eq:decay-osc-new}
    \osc_{p'} A(\nabla u)(x,\theta_0 t) \leq \theta_0^\beta
    \osc_{p'} A(\nabla u)(x,t) + c_\beta \osc_{p'} F(x, t)
  \end{align}
  holds for all $x \in B$ and $t >0$ such that $B_t(x) \subset \Omega$.

  Now, assume that $q<\infty$. We start with the representation from
  Lemma~\ref{lem:characterization} (using~$w=p'$)
  \begin{align*}
    \abs{A(\nabla u)}_{\bfB^s_{\rho,q}(B)}
    &\lesssim R^s\Bigg( \int_0^{R} \bigg(
      \frac{\normm{\osc_{p'}^B A(\nabla u)(\cdot,t)
      }_{L^\rho(B)}}{t^s}
      \bigg)^q \frac{dt}{t} \Bigg)^{\frac
      1q}.
  \end{align*}
  For every $x \in B$ and $t \in (0,R)$, we have
  $\abs{B_t(x) \cap B} \eqsim \abs{B_t(x)}$ and therefore we may write 
  $\osc_{p'}^B A(\nabla u)(\cdot,t) \lesssim \osc_{p'} A(\nabla
  u)(\cdot,t)$. This implies
  \begin{align}
      \abs{A(\nabla u)}_{\bfB^s_{\rho,q}(B)} &\lesssim R^s\Bigg(
      \int_0^{R} \bigg( \frac{\normm{\osc_{p'} A(\nabla
          u)(\cdot,t) }_{L^\rho(B)}}{t^s} \bigg)^q \frac{dt}{t}
      \Bigg)^{\frac 1q}
      \nonumber\\
       &\lesssim R^s\Bigg(
      \int_0^{\theta_0 R} \bigg( \frac{\normm{\osc_{p'} A(\nabla
          u)(\cdot,t) }_{L^\rho(B)}}{t^s} \bigg)^q \frac{dt}{t}
      \Bigg)^{\frac 1q}
      \nonumber\\
       &\qquad + R^s\Bigg(
      \int_{\theta_0 R}^R \bigg( \frac{\normm{\osc_{p'} A(\nabla
          u)(\cdot,t) }_{L^\rho(B)}}{t^s} \bigg)^q \frac{dt}{t}
      \Bigg)^{\frac 1q}
      \nonumber\\
      & =: \mathrm{I} + \mathrm{II}.
  \label{eq:decay-osc-new-2}
  \end{align}
  Now, rescaling of the integral and our decay
  estimate~\eqref{eq:decay-osc-new} imply
  \begin{align*}
    \mathrm{I} 
    &= R^s\Bigg( \int_0^{R} \bigg( \frac{\normm{\osc_{p'} A(\nabla
      u)(\cdot,\theta_0 t) }_{L^\rho(B)}}{(\theta_0 t)^s} \bigg)^q
      \frac{dt}{t} \Bigg)^{\frac 1q} 
    \\
    &\leq \theta_0^{\beta-s} R^s\Bigg( \int_0^{R} \bigg(
      \frac{\normm{\osc_{p'} A(\nabla u)(\cdot,t)
      }_{L^\rho(B)}}{t^s}
      \bigg)^q \frac{dt}{t} \Bigg)^{\frac 1q}
    \\
    &\quad 
+ c_\beta\,\theta_0^{-s}
      R^s\Bigg( \int_0^{R} \bigg( 
      \frac{\normm{\osc_{p'} F(\cdot,t)
      }_{L^\rho(B)}}{t^s}
      \bigg)^q \frac{dt}{t} \Bigg)^{\frac 1q}.
  \end{align*}
  Since $\theta_0^{\beta-s}<1$ we can absorb the
  $\int_0^{\theta_0 R} \cdot\, dt$-part of the first integral
  into~$\mathrm{I}$ to obtain
  \begin{align*}
    \begin{aligned}
      \mathrm{I} &\lesssim R^s \Bigg( \int_{\theta_0R}^{R} \bigg(
      \frac{\normm{\osc_{p'} A(\nabla u)(\cdot,t) }_{L^\rho(B)}}{t^s}
      \bigg)^q \frac{dt}{t} \Bigg)^{\frac 1q}
+ R^s\Bigg( \int_0^{R} \bigg( \frac{\normm{\osc_{p'}
          F(\cdot,t) }_{L^\rho(B)}}{t^s} \bigg)^q \frac{dt}{t}
      \Bigg)^{\frac 1q}
      \\
      &= \mathrm{II} + R^s\Bigg( \int_0^{R} \bigg( \frac{\normm{\osc_{p'}
          F(\cdot,t) }_{L^\rho(B)}}{t^s} \bigg)^q \frac{dt}{t}
      \Bigg)^{\frac 1q}.
    \end{aligned}
  \end{align*}
  For $x \in B$ and $t \in (0,R)$ we have~$\osc_{p'}F(x,t)= \osc_{p'}^{2B}
  F(x,t)$. Thus,
  \begin{align*}
    \mathrm{I} &\lesssim \mathrm{II} + R^s \Bigg( \int_0^{R} \bigg( \frac{\normm{\osc^{2B}_{p'}
                 F(\cdot,t) }_{L^\rho(2B)}}{t^s} \bigg)^q \frac{dt}{t}
                 \Bigg)^{\frac 1q}
               = \mathrm{II} + \abs{F}_{\bfB^s_{\rho,q}(2B)}.
  \end{align*}
  Combining this with~\eqref{eq:decay-osc-new-2} we obtain
  \begin{align}
    \label{eq:decay-osc-new-3}
    \abs{A(\nabla u)}_{\bfB^s_{\rho,q}(B)} 
    &\lesssim \abs{F}_{\bfB^s_{\rho,q}(2B)} + \mathrm{II}.
  \end{align}
  Moreover, for $x\in B$ and $t \in (\theta_0R, R)$ we have
  $B_t(x) \subset 2B\subset\Omega$ and $\abs{B_t(x)} \eqsim \abs{2B}$ such that 
  (using Lemma~\ref{lem:osc-aux})
  \begin{align}
    \label{eq:decay-osc-new-10}
    \osc_{p'} A(\nabla u)(x,t)
         &\lesssim \bigg(\dashint_{2B} \abs{A(\nabla u) -
                        \mean{A(\nabla u)}_{2B}}^{p'}\,dy \bigg)^{\frac 1{p'}}.
  \end{align}
  Hence,
  \begin{align*}
    \mathrm{II} &\lesssim R^s \Bigg( \int_{\theta_0R}^{R} \bigg(
                  \frac{\normm{1}_{L^\rho(B)}}{t^s}
                  \bigg)^q \frac{dt}{t} \Bigg)^{\frac 1q}\,  \bigg(\dashint_{2B} \abs{A(\nabla u) -
                  \mean{A(\nabla u)}_{2B}}^{p'}\,dy \bigg)^{\frac
                  1{p'}}
    \\
                &\lesssim \bigg(\dashint_{2B} \abs{A(\nabla u) -
                  \mean{A(\nabla u)}_{2B}}^{p'}\,dy \bigg)^{\frac 1{p'}}
  \end{align*}
  which together with~\eqref{eq:decay-osc-new-3} proves the claim for the Besov scale with $q<\infty$. The proof for $q=\infty$ follows by
  straightforward modifications. Finally, also the statement for the
  Triebel-Lizorkin scale is shown completely analogously with
  $\normm{\cdot}_{L^\rho(B)}$ and $\int \cdot\, \frac{dt}{t}$ changing
  places.

  So far we have shown our claim in the case of Banach spaces,
  i.e., if $\rho,q \geq 1$. Let us now explain the changes for the general quasi-Banach regime $0<\min\{\rho,q\}<1$. In this case additional constants might appear in
  the application of the (quasi-) triangle inequalities for the (quasi-) norms
  $\normm{\cdot}_\rho$ and $(\int |\cdot|^q \frac{dt}{t})^{\frac
    1q}$. Hence,~$\theta_0^{\beta-s}$ has to be replaced by
  $c_{\rho,q} \theta_0^{\beta-s}$. Thus, our proof still works
  if~$\theta_0$ is so small that $c_{\rho,q} \theta_0^{\beta-s} <
  1$. Unfortunately, this is not guaranteed by
  Theorem~\ref{thm:osc-estimates}. However, it immediately follows
  from Theorem~\ref{thm:osc-estimates-general} with the help
  of~\eqref{eq:osc-est-gen-2} that for every fixed~$\theta_1 \in (0,1)$
  we have
  \begin{align}
    \osc_{p'} A(\nabla u)(x,\theta_1 t) \leq c\,\theta_1^{\beta}
    \osc_{p'} A(\nabla u)(x,t) + c_{\beta,\theta_1} \osc_{p'} F(x, t).
  \end{align}
  So, overall we obtain the factor
  ~$c\, c_{\rho,q} \theta_1^{\beta-s}$ instead
  of~$\theta_0^{\beta-s}$. For small~$\theta_1$ we can still absorb the terms as in the Banach case. The price to pay is a larger factor in front of the $F$~terms. Anyhow, this proves the general case.
\end{proof}

\subsection{Transfer to~\texorpdfstring{$\nabla u$}{nabla u} and~\texorpdfstring{$V(\nabla u)$}{V(nabla u)}}
\label{ssec:transfer-nabla-u}

In this section we show how to transfer the regularity statements for~$A(\nabla u)$ to~$\nabla u$ and $V(\nabla u)$. 
To this end, for fixed $\alpha>0$ let us define a transformation $T_\alpha$ of arbitrary vectors or matrices~$Q$ by
\begin{align*}
  Q \mapsto T_\alpha(Q) &:=  \abs{Q}^\alpha \frac{Q}{\abs{Q}}.
\end{align*} 
Then under composition $\{T_\alpha \sep \alpha>0\}$ forms a group (with identity $T_1$ and inverse $T_\alpha^{-1}=T_{\frac{1}{\alpha}}$, $\alpha>0$). In particular, for $\alpha,\beta>0$ there holds $T_{\alpha\beta}(Q)=T_\alpha(T_\beta(Q))$ and hence
\begin{align}\label{eq:T-representation}
 \nabla u = T_{\frac{2}{p}}(V(\nabla u))= T_{\frac 1{p-1}}(A(\nabla u)), 
 \quad \text{as well as} \quad
  V(\nabla u) = T_{\frac{p'}{2}}(A(\nabla u)).
\end{align}
In our situation ($p\geq 2$), we have $\frac{2}{p}, \frac{1}{p-1}, \frac{p'}{2} \in (0,1]$. Therefore, the subsequent proposition is of fundamental importance to us.
\begin{proposition}
  \label{pro:power}
  Let $B\subseteq \setR^d$ denote some ball and assume $\alpha \in (0,1]$. 
  \begin{enumerate}
  \item If $0<r \leq \infty$, then we have
    \begin{align*}
      \norm{T_\alpha (G)}_{L^{r/\alpha}(B)}
      &=\norm{G}_{L^{r}(B)}^\alpha.
    \end{align*}
  \item If $s$, $\rho$, $q$, and $w$ satisfy the conditions of Lemma~\ref{lem:characterization}, then 
    \begin{align*}
      \abs{T_\alpha(G)}_{\bfB^{\alpha s}_{\rho/\alpha,q/\alpha}(B)}
      \lesssim
      \abs{G}_{\bfB^s_{\rho,q}(B)}^\alpha.
    \end{align*}
    Moreover, the same is true when the $\bfB$ spaces are replaced by $\bfF$ spaces.
  \end{enumerate}
\end{proposition}
We note in passing that Proposition~\ref{pro:power}(a) also holds for the respective scaling invariant norms. Thus, Proposition~\ref{pro:power}(b) can be used to show that $G \in \bfB^s_{\rho,q}(B)$ implies $T_\alpha(G) \in \bfB^{\alpha s}_{\rho/\alpha,q/\alpha}(B)$ and likewise for the $\bfF$-case. In fact, also the stated bounds remain true if the (quasi-) semi norms are replaced by the corresponding full (quasi-) norms. A similar statement for the scalar case is contained in~\cite[Section 5.4]{RunSic96}. 
However, the vectorial setting is different. Therefore, below we will present a general but quite simple proof based on the representation in Lemma~\ref{lem:characterization}. Before we get to this proof, we need
an auxiliary lemma on oscillations.
\begin{lemma}
  \label{lem:powers-aux}
  Let $\alpha \in (0,1]$ and $1\leq w <\infty$. Then for all balls $B$ and $G \in
  L^w(B)$, there holds
  \begin{align*}
    \bigg(\dashint_B \abs{T_\alpha (G)-\mean{T_\alpha
    (G)}_B}^w\,dx\bigg)^{\frac 1w}
    &\lesssim 
      \bigg(\dashint_B \abs{G-\mean{G}_B}^{\alpha w}\,dx\bigg)^{\frac 1w}. 
  \end{align*}
\end{lemma}
\begin{proof}
  Recall that $T_\alpha(Q) = \abs{Q}^{\alpha} \frac{Q}{\abs{Q}}$. Thus
  $T_\alpha(Q)=A(Q)$ if we redefine our exponent~$p$ just for this
  proof (!) as~$p:= \alpha+1 \in (1,2]$. Then it follows from
  Lemma~\ref{lem:hammer} (with this $p$) that for
  all~$P,Q$
  \begin{align}
    \label{eq:powers-aux1}
      \abs{T_\alpha(P) - T_\alpha(Q)} 
      &= \abs{A(P) - A(Q)} \nonumber\\
      &\eqsim \phi_{\abs{Q}}'(\abs{P-Q})
      \nonumber\\
      &\eqsim (\abs{Q} + \abs{P-Q})^{\alpha-1} \abs{P-Q} 
      \nonumber\\
      &\leq \abs{P-Q}^\alpha.
  \end{align}
  Using Lemma~\ref{lem:osc-aux} and the bijectivity of
  $T_\alpha$, and~\eqref{eq:powers-aux1} we estimate
  \begin{align*}
    \bigg(\dashint_B \abs{T_\alpha (G)-\mean{T_\alpha
    (G)}_B}^w\,dx\bigg)^{\frac 1w}
    &\eqsim \inf_{H_0} \bigg(\dashint_B \abs{T_\alpha (G)- H_0}^w\,dx\bigg)^{\frac 1w}
    \\
    &= \inf_{G_0} \bigg(\dashint_B \abs{T_\alpha (G)-T_\alpha
      (G_0)}^w\,dx\bigg)^{\frac 1w}
    \\
    &\lesssim \inf_{G_0} \bigg(\dashint_B \abs{G-G_0}^{\alpha w}
      \,dx\bigg)^{\frac 1w}
      \\
    &\eqsim \bigg(\dashint_B \abs{G-\mean{G}_B}^{\alpha
      w} \,dx\bigg)^{\frac 1w},
  \end{align*}
  where the infimum is taken over all constants~$G_0$, resp.\ $H_0$.
\end{proof}
\begin{remark}
  \label{rem:powers-aux}
  Note that in Lemma~\ref{lem:powers-aux} it is possible to
  replace~$B$ by $B \cap B_t(x)$ for each~$x \in B$ and all $t \in
  (0,R]$, where~$R$ denotes the radius of~$B$.
\end{remark}
We are now prepared to prove Proposition~\ref{pro:power}.
\begin{proof}[Proof of Proposition~\ref{pro:power}]
  The formula
  $\norm{T_\alpha (G)}_{L^{r/\alpha}(B)} =
  \norm{G}_{L^{r}(B)}^\alpha$ is obvious for all $r\in(0,\infty]$. 
  So, let us now show part (b), i.e.,
  $\abs{T_\alpha(G)}_{\bfB^{\alpha s}_{\rho/\alpha,q/\alpha}(B)}
  \lesssim \abs{G}_{\bfB^s_{\rho,q}(B)}^\alpha$.
  For this purpose, we will use the characterization of
  Lemma~\ref{lem:characterization}.  
  It follows from Lemma~\ref{lem:powers-aux} and
  Remark~\ref{rem:powers-aux} that
  \begin{align}
    \label{eq:oscTalpha}
    \osc_{w/\alpha}^B(T_\alpha (G))(x,t)
    &\lesssim
      \big( \osc_{w}^B G(x,t) \big)^\alpha.
  \end{align}
  Moreover, $s$, $\rho$, $q$, $w$ replaced by
  $\alpha s$, $\rho/\alpha$, $q/\alpha$, $w/\alpha$ also satisfy
  the conditions of Lemma~\ref{lem:characterization}.  Thus, we can
  calculate
  \begin{align*}
    \abs{T_\alpha (G)}_{\bfB^{\alpha s}_{\rho/\alpha,q/\alpha}(B)}
	&=
      R^s\Bigg( \int_0^R \bigg(
      \frac{\normm{\osc_{w/\alpha}^B T_\alpha (G)(\cdot,t)
      }_{L^{\rho/\alpha}(B)}}{t^{\alpha s}}
      \bigg)^{\frac q \alpha} \frac{dt}{t} \Bigg)^{\frac \alpha q}
    \\
    &\lesssim
      R^s\Bigg( \int_0^R \bigg(
      \frac{\normm{ \big(\osc_{w}^B G(\cdot,t)
      \big)^\alpha }_{L^{\rho/\alpha}(B)}}{t^{\alpha s}}
      \bigg)^{\frac q \alpha} \frac{dt}{t} \Bigg)^{\frac \alpha q}
    \\
    &=    
      R^s\Bigg( \int_0^R \bigg(
      \frac{\normm{\osc_{w}^B G(\cdot,t)
      }_{L^{\rho}(B)}}{t^{s}}
      \bigg)^q \frac{dt}{t} \Bigg)^{\frac \alpha q}
    \\
    &=
      \abs{G}_{\bfB^{s}_{\rho,q}(B)}^\alpha.
  \end{align*}
  This proves the $\bfB^s_{\rho,q}$-estimate (with the obvious
  modifications for~$q=\infty$).  The $\bfF^s_{\rho,q}$-case is shown
  analogously.
\end{proof}

We can now combine Theorem~\ref{thm:reg-transfer} and
Proposition~\ref{pro:power} with the representations~\eqref{eq:T-representation} to conclude new regularity results for~$\nabla u$ and~$V(\nabla u)$.

\begin{corollary}
  \label{cor:reg-transfer-u}
  Under the assumptions of Theorem~\ref{thm:reg-transfer} there holds
  \begin{align*}
    \norm{\nabla u}_{L^{\rho(p-1)}(B)}^{p-1} 
    &=\norm{V(\nabla u)}_{L^{2 \rho/ p'}(B)}^{\frac{2}{p'}}
      = 
      \norm{A(\nabla u)}_{L^\rho(B)}
    \\
    \intertext{and}
     \abs{\nabla u}_{\bfB^{\frac{s}{p-1}}_{\rho (p-1),q(p-1)}(B)}^{p-1}
     &\lesssim 
        \abs{V(\nabla u)}_{\bfB^{\frac{p's}{2}}_{\frac{2\rho}{p'},\frac{2q}{p'}}(B)}^{\frac{2}{p'}} 
     \\
     &\lesssim
        \abs{A(\nabla u)}_{\bfB^s_{\rho,q}(B)}
     \\
    &\lesssim
      \abs{F}_{\bfB^s_{\rho,q}(2B)} + \bigg(
      \dashint_{2B} \abs{A(\nabla u)
      - \mean{A(\nabla u)}_{2B}}^{p'}\,dx \bigg)^{\frac{1}{p'}}. 
  \end{align*}
  Under the same additional assumptions as in
  Theorem~\ref{thm:reg-transfer} the latter estimates remain true in the scale of
  Triebel-Lizorkin spaces.
\end{corollary}

\begin{remark}
  \label{rem:compare-CGP}
  Let us compare Corollary~\ref{cor:reg-transfer-u} to the results from~\cite{CloGioPas17}.  They
  prove for $p\geq 2$, $d \geq 2$, $s \in (0,1)$, and
  $1 \leq q \leq \frac{2 d}{d-2s}$ that locally
  \begin{align*}
    F \in \bfB^s_{2,q} \qquad \text{implies} \qquad V(\nabla u) \in
    \bfB^{\frac{p' s}{2}}_{2,\frac{2q}{p'}}.
  \end{align*}
  Our result applied to the same
  situation ($\rho=2$ in dimension $d=2$) yields that 
  \begin{align*}
    F \in \bfB^s_{2,q} \qquad \text{implies} \qquad V(\nabla u) \in
    \bfB^{\frac{p' s}{2}}_{\frac{4}{p'},\frac{2q}{p'}}.
  \end{align*}
  In particular, the integrability of~$V(\nabla u)$ is increased from
  $2$ to $\frac{4}{p'}$ and we need no restrictions on the fine index~$q$.
\end{remark}

\begin{remark}
  \label{rem:compare-DDHSW}
  Let us compare our results also to the findings in~\cite{DahDieHarSchWei16}. They
  have studied the Besov regularity of~$u$ measured in the $L^p$
  adaptivity scale, i.e., those~$\bfB^\sigma_{\tau,\tau}$ with
  $\sigma-\frac{d}{\tau} = -\frac{d}{p}$. For~$p \geq 2$, Lipschitz
  domains in $d=2$, and $f:=-\mathrm{div} F \in L^\infty$ they show
  that globally $u \in \bfB^\sigma_{\tau,\tau}$ for
  all~$\sigma \in (0,p')$. However, for~$f \in L^\rho$ with
  $\rho \in (2p,\infty)$ they obtain the condition
  $\sigma < 1+\frac{1-2/\rho}{p-1}$. So their upper bound
  for~$\sigma$ depends on~$\rho$. The reason for this is that the
  lower integrability of~$f$ induces a smaller exponent
  $\alpha_\rho^* = \frac{1-2/\rho}{p-1}$ of local H\"older
  continuity for~$\nabla u$. Hence, with the techniques
  from~\cite{DahDieHarSchWei16} it is only possible to treat
  differentiabilities up to~$s < \alpha_\rho^*$. Moreover, for
  $f \in L^\rho$ with $\rho \in (2,2p]$ they
  are restricted to~$\sigma < 1 + \frac 1p$ which is again non optimal.  However,
  our interior oscillations estimates do not require the use of
  H\"older spaces.  Thus, this unnatural bound does not appear. 
  To see this, let us now assume that~$f \in L^\rho$ with $\rho \geq p'$
  and~$d=2$. Using the \Bogovskii{} operator we find
  $F \in \bfF^1_{\rho,2}$ with $\divergence F =f$.  
  So $F \in \bfB^s_{p',q}$ for any $s < 1$ and $q\in (0,\infty]$ such
  that from Corollary~\ref{cor:reg-transfer-u} it follows that locally
  in the interior there holds
  $u \in \bfB^{1+\frac{s}{p-1}}_{p,q}$. Note that independently of $\rho$ this yields $u \in \bfB^\sigma_{\tau,\tau}$ for
  all~$\sigma \in (0,p')$ which improves~\cite{DahDieHarSchWei16}.
\end{remark}


\def\cprime{$'$}

\end{document}